%% file: main_arxiv.tex
\documentclass{article}

\usepackage[top=1in, bottom=1in, left=1in, right=1in]{geometry}

\usepackage[T1]{fontenc}
\usepackage{xcolor} 
\usepackage{mathtools} 
\usepackage{amsthm}
\usepackage{orcidlink}

\input{_macros}

\usepackage{authblk}
\usepackage{etoc}
\usepackage{wrapfig}
\usepackage{bm}
\allowdisplaybreaks
\newtheorem{informalresult}{Informal Result}

\newcommand\iidsim{\stackrel{\mathclap{i.i.d.}}{\sim}}

\makeatletter
\renewcommand{\paragraph}{%
  \@startsection{paragraph}{4}%
  {\z@}{1.5ex \@plus 1ex \@minus .2ex}{-1em}%
  {\normalfont\normalsize\bfseries}%
}
\makeatother

\title{BBP Phase Transition for a Doubly Sparse Deformed Model}

\author[1]{Ioana Dumitriu\thanks{Email: \texttt{idumitriu@ucsd.edu}}}
\author[1]{JD Flynn\thanks{Email: \texttt{jdflynn@ucsd.edu}}\orcidlink{0009-0007-5345-0135}}
\author[2]{Zhichao Wang\thanks{Email: \texttt{zhichao.wang@berkeley.edu}}}

\affil[1]{University of California, San Diego}
\affil[2]{University of California, Berkeley}

\date{} 

\begin{document}
\maketitle

\begin{abstract}
\input{sections/abstract}
\end{abstract}


\tableofcontents

\input{sections/introduction}

\input{sections/model}

\input{sections/main_results}

\input{sections/experiments}

\input{sections/future_directions}

\input{sections/preliminary}

\input{sections/distinguish}

\input{sections/recover}

\paragraph{Acknowledgments} This project was supported by the following National Science Foundation (NSF) grant: DMS-2154099. Special thanks to Yizhe Zhu and Yiyun He for their valuable suggestions and comments on Lemma \ref{lemma_2}.


\bibliographystyle{alpha} 
\bibliography{ref.bib}       


\appendix
\section{Appendix}
\input{sections/appendix}
\end{document}

%% file: _macros.tex
\usepackage{comment,url,algorithm,algorithmic,graphicx,subcaption,relsize}
\usepackage{amssymb,amsfonts,amsmath,amsthm,amscd,dsfont,mathrsfs,mathtools,nicefrac}
\usepackage{float,psfrag,epsfig,color,xcolor,url,hyperref}
\usepackage{epstopdf,bbm,enumitem}
\usepackage[toc,page]{appendix}
\usepackage[mathscr]{euscript}

\def\ddefloop#1{\ifx\ddefloop#1\else\ddef{#1}\expandafter\ddefloop\fi}

\def\ddef#1{\expandafter\def\csname bb#1\endcsname{\ensuremath{\mathbb{#1}}}}
\ddefloop ABCDEFGHIJKLMNOPQRSTUVWXYZ\ddefloop

\def\ddef#1{\expandafter\def\csname c#1\endcsname{\ensuremath{\mathcal{#1}}}}
\ddefloop ABCDEFGHIJKLMNOPQRSTUVWXYZ\ddefloop

\def\ddef#1{\expandafter\def\csname v#1\endcsname{\ensuremath{\boldsymbol{#1}}}}
\ddefloop ABCDEFGHIJKLMNOPQRSTUVWXYZabcdefghijklmnopqrstuvwxyz\ddefloop

\def\ddef#1{\expandafter\def\csname v#1\endcsname{\ensuremath{\boldsymbol{\csname #1\endcsname}}}}
\ddefloop {alpha}{beta}{gamma}{delta}{epsilon}{varepsilon}{zeta}{eta}{theta}{vartheta}{iota}{kappa}{lambda}{mu}{nu}{xi}{pi}{varpi}{rho}{varrho}{sigma}{varsigma}{tau}{upsilon}{phi}{varphi}{chi}{psi}{omega}{Gamma}{Delta}{Theta}{Lambda}{Xi}{Pi}{Sigma}{varSigma}{Upsilon}{Phi}{Psi}{Omega}{ell}\ddefloop

\def\balign#1\ealign{\begin{align}#1\end{align}}
\def\baligns#1\ealigns{\begin{align*}#1\end{align*}}
\def\balignat#1\ealign{\begin{alignat}#1\end{alignat}}
\def\balignats#1\ealigns{\begin{alignat*}#1\end{alignat*}}
\def\bitemize#1\eitemize{\begin{itemize}#1\end{itemize}}
\def\benumerate#1\eenumerate{\begin{enumerate}#1\end{enumerate}}

\newenvironment{talign*}
{\csname align*\endcsname}
{\endalign}
\newenvironment{talign}
{\csname align\endcsname}
{\endalign}

\def\balignst#1\ealignst{\begin{talign*}#1\end{talign*}}
\def\balignt#1\ealignt{\begin{talign}#1\end{talign}}

\let\originalleft\left
\let\originalright\right
\renewcommand{\left}{\mathopen{}\mathclose\bgroup\originalleft}
\renewcommand{\right}{\aftergroup\egroup\originalright}

\def\tinycitep*#1{{\tiny\citep*{#1}}}
\def\tinycitealt*#1{{\tiny\citealt*{#1}}}
\def\tinycite*#1{{\tiny\cite*{#1}}}
\def\smallcitep*#1{{\scriptsize\citep*{#1}}}
\def\smallcitealt*#1{{\scriptsize\citealt*{#1}}}
\def\smallcite*#1{{\scriptsize\cite*{#1}}}

\def\mbb#1{\mathbb{#1}}

\theoremstyle{plain}

\def\<{\left\langle} 
\def\>{\right\rangle}

\def\iff{\Leftrightarrow}

\def\Re{\operatorname{Re}} 

\def\P{\mbb{P}} 

\DeclareSymbolFont{rsfs}{U}{rsfs}{m}{n}
\DeclareSymbolFontAlphabet{\mathscrsfs}{rsfs}
\ifdefined\nonewproofenvironments\else
\ifdefined\ispres\else
\newtheorem{theorem}{Theorem}
\newtheorem{lemma}[theorem]{Lemma}
\newtheorem{corollary}[theorem]{Corollary}

\theoremstyle{definition}
\newtheorem{definition}[theorem]{Definition}

\renewenvironment{proof}{\noindent\textbf{Proof.}\hspace*{.3em}}{\qed\\}
\newenvironment{proof-sketch}{\noindent\textbf{Proof Sketch}
	\hspace*{0.3em}}{\qed\\}
\newenvironment{proof-idea}{\noindent\textbf{Proof Idea}
	\hspace*{0.3em}}{\qed\\}
\newenvironment{proof-of-lemma}[1][{}]{\noindent\textbf{Proof of Lemma {#1}.}
	\hspace*{0.3em}}{\qed\\}
\newenvironment{proof-of-theorem}[1][{}]{\noindent\textbf{Proof of Theorem {#1}.}
	\hspace*{0.3em}}{\qed\\}
\newenvironment{proof-of-coro}[1][{}]{\noindent\textbf{Proof of Corollary {#1}.}
	\hspace*{0.3em}}{\qed\\}
\fi

\newtheorem{proposition}[theorem]{Proposition}
\newtheorem{claim}[theorem]{Claim}
\newtheorem{assumption}{Assumption}

\fi
\makeatletter
\@addtoreset{equation}{section}
\makeatother

\hypersetup{
colorlinks,
linkcolor={red!50!black},
citecolor={blue!50!black},
urlcolor={blue!80!black}
}



%% file: sections/abstract.tex

We prove the equivalent of the Baik, Ben Arous, P\'ech\'e \cite{baik2004phasetransitionlargesteigenvalue} phenomenon for a novel, doubly sparse model where both the Wigner noise matrix and signal vector(s) are sparse. Specifically, we consider a deformed sub-Gaussian sparse Wigner ensemble with a fixed number of sub-Gaussian spike vectors of the same-order sparsity added. We show that spike vectors with signals greater than one are correlated with the top eigenvectors of the deformed ensemble and that each spike vector of signal greater than one induces an outlier eigenvalue. Notably, our results hold in the supercritical sparsity regime for the Wigner matrix ($q \gg \frac{\log n}{n}$) and for any sparse spike vector with an unbounded number of entries ($np\to \infty$). No further relationship between the sparsities of the noise matrix ($q$) and spike vectors ($p$) is necessary. This generalizes the work of Benaych-Georges and Nadakuditi \cite{benaychgeorges2010eigenvalues} and P\'ech\'e \cite{péché2005largesteigenvaluesmallrank}. 

%% file: sections/introduction.tex
\section{Introduction}

\subsection{Overview}

\paragraph{Spiked Random Matrix Models.} 

In the modern age of gigantic data sets, \textit{principal component analysis} (PCA) has proven itself to be the method of choice for data scientists and statisticians to extract the largest eigenvector from covariance matrices to enact dimensional reduction and improve interpretability, all while maintaining integral structural information.  While deformed matrix models had been explored previously, Johnstone \cite{johnstone2001} formalized the problem in the context of statistics and introduced the ``spike + noise'' structure: a low-rank spike, represented by a vector $v$ with signal strength $\theta$, is added to a random noise matrix $W$. The model's application and practical value has since permeated throughout a variety of data-driven fields. Biologists apply the model on matrices containing genetic information to extract relevant columns through a process called biclustering, thus learning shared phylogenetic traits between species by identifying and comparing similar components in their gene sequences \cite{NIPS2011_c6e19e83}. In an attempt to solve the angular synchronization problem and thus obtain an accurate estimation for a set of unknown angles from a collection of noisy measurements, one can employ PCA through the construction and analysis of a specially designed Hermitian matrix \cite{singer2009}. A graph theory problem of great interest sees the spike model as an adjacency matrix of a random weighted graph with a planted clique and aims to extract the latter \cite{montanari2014limitationspectralmethodsgaussian}. These applications are but a few of the many instances where the model provides fruitful results.

Prior to Johnstone's formalization of the model, a series of works in random matrix theory laid the foundation for the spectral properties of $W$ sampled from various distributions (we refer the reader to \cite{anderson2010introduction} for a comprehensive treatment of the literature). Armed with rich results covering the pure noise case ($\theta =0$), henceforth referred to as the ``null model,'' Johnstone and further works considered the natural approach of balancing the random spectral properties of the noise matrix against the deterministic eigenvalue of the low rank signal to solve two distinct problems: 



\begin{itemize}
    \item \textbf{Distinguishability: } Can we \textit{distinguish} between the planted and null models with high probability?
    \item \textbf{Recovery: } Can we  \textit{recover} a vector that correlates with the true  planted signal $v$?
\end{itemize}

For both problems, we seek a phase transition, or \textit{information-theoretical threshold}, where a specific balance of parameter values controls the solvability of the problem. The thresholds can differ based on whether we aim to distinguish or recover. Moreover, the task of recovering the planted signal splits into varying levels of recovery depending on the strength of correlation we can achieve or fraction of entries we can recover. For instance, can we fully recover the entries of the spike matrix down to the true value, or simply produce a vector with a non-trivial correlation with the true signal?\footnote{In the present paper, we consider the latter, which we define as weak recovery in Definition \ref{weaklyrecoverdef}.} 

\paragraph{BBP Phase Transition.} The prominent spectral method for solving these problems arises from the power method: we use the leading eigenvalue $\lambda_1(X)$ to reliably distinguish between the planted and null models, and we use the corresponding leading eigenvector $u_1(X)$ as a recovery candidate for the signal. The celebrated ``BBP transition'' of Baik, Ben Arous, and P\'ech\'e \cite{baik2004phasetransitionlargesteigenvalue} establishes a sharp phase transition for the \textbf{Spiked (Gaussian\footnote{\cite{baiksilverstein} drops the Gaussianity condition, requiring instead only a  broad finite fourth moment assumption.}) Wishart Model}: given $\gamma>0$ and $N\in\mathbb N$, we observe $X=\frac{1}{N}YY^T$ where $Y$ is an $n\times N$ matrix with i.i.d. columns drawn from $\mathcal{N}(0,I_n+\gamma v v^T)$; in their work, a threshold arises for the emergence of the spike from the Marchenko-Pastur bulk. 
Shortly after, P\'ech\'e \cite{péché2005largesteigenvaluesmallrank, F_ral_2007} and F\'eral proved an analogous result for the \textbf{Spiked Wigner Model}: we observe \begin{equation}\label{eq:spiked_wigner}
    X=\frac{\theta }{n} v v^T+\frac{1}{\sqrt{n}}W
\end{equation} 
where the symmetric matrix $W$ has independent entries of $\mathbb E[W_{ii}]=\mathbb E[W_{ij}]=0,$ $ \mathbb E[W_{ij}^2]=\sigma^2$, and $\mathbb E[W_{ii}^2]$ bounded. They showed that, with some additional conditions on the moments of the matrix, a threshold arises for the emergence of the spike from the semicircular bulk:
\begin{informalresult}{(\cite{péché2005largesteigenvaluesmallrank}, \cite{F_ral_2007}, \cite{benaychgeorges2010eigenvalues})} \label{general_bbp}
    Given unit vector $v$ and $\theta >0$, we find, as $n\to\infty$, 
    \begin{itemize}
        \item if $\theta \le 1$, then $\lambda_1(X)\xrightarrow[]{a.s.}2$, and $\langle u_1(X), v \rangle^2 \xrightarrow[]{a.s.} 0$; 
        \item if $\theta >1$, then $\lambda_1(X)\xrightarrow[]{a.s.}\theta +\frac{1}{\theta }>2$, and $\langle u_1(X), v \rangle^2 \xrightarrow[]{a.s.} 1-\frac{1}{\theta ^2}>0$~.
    \end{itemize}
\end{informalresult}

In order to paint a general picture of the phenomenon, the above informal result omits specific conditions on the distribution of the matrices involved. Furthermore, the result can extend to the case of adding multiple spikes. For instance, \cite{péché2005largesteigenvaluesmallrank} considered $r\ge1$ spikes, but did so by representing the spikes as one rank-$r$ diagonal matrix due to the rotational invariance of their Gaussian Unitary Ensemble noise matrix. \cite{benaychgeorges2010eigenvalues} explored a wider class of perturbations by requiring only the spike or noise matrix to be orthogonally invariant, meaning the distribution of the matrix is invariant under the action of the orthogonal group under conjugation. In our model, we do not require rotational invariance of the spike matrix or the noise. 

Throughout the study of spiked models, many results hold in similar fashion for both the Wishart and Wigner models. In fact, the computational equivalence of the two models is widely conjectured, with \cite{bresler2025computationalequivalencespikedcovariance} providing the most recent significant progress in grouping the two statistical models into an equivalence class. In this present work, we consider only the spiked Wigner model.

\paragraph{BBP Phenomenon with Altered Parameters.} Naturally, one may consider altering the model parameters to explore if the BBP spectral method still remains a viable solution for distinguishing and recovering. Initially, researchers sought to relax the conditions on the assumptions of the noise matrix $W$. \cite{F_ral_2007} first proved the BBP transition for Wigner matrices with symmetric laws ($\mathbb E[W_{ij}^{2k+1}]=0$ and even sub-Gaussian moments). \cite{Capitaine_2009} extended the result to models satisfying a Poincair\`e inequality, and \cite{pizzo2011finiterankdeformationswigner} followed after with an extension to Wigner models with finite fifth moments. More recent works have also made progress on relaxing the assumptions on the model parameters, but many have also worked toward altering the structure of the model completely. First introduced by \cite{Lesieur2015} and rigorously analyzed by \cite{perry2016optimalitysuboptimalitypcaspiked}, one can consider an entrywise transformation of the spiked matrix model. This transformed model draws inspiration from applications of Improved PCA and Deep Neural Networks with the intention of amplifying the Signal-to-Noise Ratio (SNR) $\theta$, thus making the signal more detectable. Further efforts by \cite{feldman2025spectral, guionnet2023spectral} established a BBP phase transition for the transformed model, with \cite{lee2025fluctuationslargesteigenvaluestransformed} providing the best known result in terms of parameter assumptions. A series of works \cite{Guionnet_2025, pak2023optimalalgorithmsinhomogeneousspiked, NEURIPS2024_a96368eb, de2025computationalstatisticallowerbounds, mergny2024spectralphasetransitionoptimal, behne2022fundamentallimitsrankonematrix, chen2022statisticalinferencefiniteranktensors} also consider the spiked problem with inhomogeneous noise, where the signal is observed through different noisy channels across the matrix, thus accounting for a wider regime of models such as the degree-corrected Stochastic Block Model. A collection of those previous works overlaps with a separate extension of the problem where the spike and noise are considered to be tensors rather than matrices.

\paragraph{Sparse PCA.} A common setting alteration, and the model deviation most pertinent to this present work, considers sparse PCA where sparsity parameters, such as independent Bernoulli random variables, are applied entrywise to the spike vectors. In an application sense, this sparsity provides practical benefits in a high-dimensional setting, often improving interpretability, increasing efficiency, and reducing model complexity \cite{Zou2006SparsePCA,johnstone2009consistency}. Under certain model assumptions, the general BBP threshold described in Informal Result \ref{general_bbp} captures sparse PCA as the theorem assumes only a unit norm condition for the spike $v$; however, when considering alterations to the model's parameters, Informal Result \ref{general_bbp} does begin to fail. For instance, \cite{benaychgeorges2010eigenvalues} covers both the eigenvector and eigenvalue result for multiple sparse spikes, but their analysis requires either the spike or noise to be orthogonally invariant. Seeing as orthogonal invariance is broken with the introduction of sparsity in the spike, this requires the noise matrix $W$ to be orthogonally invariant. 
Methods for sparse PCA below the spectral threshold $\theta_i<1$ have been studied extensively as well (see Section \ref{beyondbbp}). 
With the landscape of sparse PCA (under certain prior assumptions) being covered by BBP, we naturally consider the possibility of introducing sparsity across all aspects of the model, which brings us to the central question of this work: 

\vspace{-1.mm}
\begin{center}
    \textit{
    Can we generalize this BBP theorem in the setting where both spikes and noise are sparse?}
\end{center}

\paragraph{Doubly Sparse PCA.}\label{doub_sparse_intro} Here, we answer the question in the affirmative. We importantly note how the BBP result of \cite{benaychgeorges2010eigenvalues} no longer holds even in the case where the nonzero entries of the noise are Gaussian, since invariance is required by their procedure and sparsity disrupts this property. When considering a candidate for recovery, many of the aforementioned BBP results either exclude the problem of recovery entirely and focus on distinguishability via eigenvalues, rely specifically on the orthogonal invariance property (similar to the results of \cite{benaychgeorges2010eigenvalues}), or apply a different technique based on stricter assumptions for the spike prior. Here, we consider an approach that does not require invariance\footnote{Circumventing orthogonal invariance to prove recovery is not uncharted terrain; for instance,  \cite{barbier_2025} explores a statistical physics approach of recovery below the spectral threshold in the regime where the noise is rotationally invariant, but the spike is not. Our model, however, allows both the noise and spike to not be rotationally invariant.}, and, to the best of our knowledge, we provide the first rigorous argument for recovering multiple overlapping spiked vectors via PCA in the setting where orthogonal invariance is not present in the noise matrix. 

\subsection{Doubly Sparse Literature Review} \label{doub_sparse_lit_review}
\paragraph{Random Matrix Theory Tools. } The results of this present work are timely and could not have been produced if not for a series of recent contributions to the field of random matrix theory. First, a major ingredient involves the control of the operator norm for the sparse Wigner matrix. The study of extremal eigenvalues of sparse matrices is a field that has recently grown significantly; however, these works either consider cases with optimal sparsity ($q\gtrsim \frac{\log n}{n}$) but bounded random variable entries \cite{tikhomirov2019outliersspectrumsparsewigner, alt2021completelydelocalizedregionerdhosrenyi} or sub-Gaussian entries but with sub-optimal sparsity ($q\gtrsim\frac{(\log n)^4}{n}$) \cite{Bandeira_2023}. \cite{augeri2023largedeviationslargesteigenvalue} provides the essential ingredient needed to show the top eigenvalue of our sub-Gaussian, sparse ($q\gg \frac{\log n}{n}$) Wigner matrix is bounded with high probability. Further, we require a local law result for our sparse Wigner matrix, which has been thoroughly explored over recent years \cite{erdHos2012spectral,tran2013sparse,erdHos2013spectral,dumitriu2019sparse,he2021fluctuations}; we adapt the local law of \cite{sparsehanson} for own purposes as their result covers, to the best of our knowledge, the most optimal $\frac{\text{polylog}(n)}{n}$ regime. 

\paragraph{Similar Sparse Models.} We remark that, while our contributions and consideration of spectrally distinguishing and recovering a sparse spike from sparse noise are novel, the notion of a sparse noisy observation has been around for a while. Indeed, Robust PCA explores the problem of separating a low-rank structure from sparse noise. \cite{candes2009robustprincipalcomponentanalysis} provides a convex optimization procedure for separating these two entities, though the only randomness considered is the location of the sparsity in the mask. Their work gives an extensive background on the many applications of Robust PCA, including video surveillance, facial recognition, ranking and collaborative filtering, and more, showing how sparse noise grows increasingly prevalent in our large-dimensional modern age with instances of occlusion, tampering, and sensor failure. These aforementioned applications naturally extend to the setting of our model, where the spike also has a sparse structure. Moreover, doubly sparse applications have also been studied across different fields. \cite{ZARMEHI2018145, sadrizadeh2019fast} provide an electrical engineering approach where they view the problem as a non-convex, NP-hard optimization problem on rectangular matrices with undefined sparsity to solve applications involving impulsive noise (image, audio, video denoising).  More recently, \cite{Han_2025} looks at a comparable doubly sparse structure, but for non-Hermitian matrices, thus changing the landscape of the problem and relying on a completely different set of random matrix theory tools. Similarly, \cite{hachem2026extreme} studies the extreme eigenvalues of an $n\times n$ non‑Hermitian sparse i.i.d. matrix with sparsity $q$ and a deterministic finite‑rank additive deformation; they show that eigenvalue outliers outside the unit disk asymptotically match the eigenvalues of finite‑rank additive deformation if $qn\to\infty$. Moreover, 
in the rank‑one case \cite{hachem2026extreme} provides control over the right‑eigenvector alignment similar to our Theorem~\ref{eigenvector_theorem} but under an additional sparsity condition $qn \gg \log^9 n$ while we only require $qn \gg \log(n)$. 
\cite{adomaityte2025pcarecoverythresholdslowrank} examines a broad class of sparse graph-based models where the noise is drawn from the configuration model with a general degree distribution with finite mean and bounded maximal degree. Via the replica method, the authors find a BBP-like phase transition for the expected location of the largest eigenvalue and expected correlation of the largest eigenvector. Our model does not overlap with any of the models considered there. 

\paragraph{Application: Planted Clique Model.}
Beyond the above applications of a doubly sparse matrix model, a major motivating example for this work is the
\textbf{Planted Clique Problem (PCP)}: given a graph $\mathcal G$, distinguish whether $\mathcal G\sim \mathcal G(n,1/2)$ or
$\mathcal G$ is drawn from Erd\'os-R\'enyi model $\mathcal G(n,1/2)$ with an additional planted clique of size $k$.
Under the null, the maximum clique size of $\mathcal G(n,1/2)$ is $(2+o(1))\log_2 n$ w.h.p.~\cite{bollobas1976cliques}, yielding an information-theoretic
detection threshold around $k\approx 2\log_2 n$ via exhaustive search.
In contrast, the best known polynomial-time algorithms (e.g., spectral/SDP methods) succeed only for $k=\Omega(\sqrt n)$ \cite{alon1998finding,feige2000finding}, and classical negative results together with modern low-degree/Sum-of-Squares analyses provide evidence that the regime $k=o(\sqrt n)$ is computationally hard \cite{jerrum1992large,chen2025almost,barak2019nearly}. Under the same regime, \cite{gamarnik2019landscape} also establishes that the computational hardness arises from the Overlap Gap Property, showing that the landscape of PCP is fractured into disconnected clusters separated by an empty region of overlap, creating a topological barrier for efficient algorithms. From the high-dimensional statistics viewpoint, PCP is also a canonical source of conjectured computational hardness
for sparse PCA: \cite{berthet2013complexity} gives polynomial-time reductions from planted clique to sparse PCA detection, and \cite{ma2015sum} shows Sum-of-Squares lower bounds for sparse PCA and further connects these problems to moment-based relaxations. While PCP is typically formulated for constant edge density and unweighted adjacency matrices, our model permits
$n$-dependent sparsity in the noise mask and real-valued edge weights, linking the planted-clique viewpoint to
sparse-PCA and spiked random matrix models. This is similar with the setting of \cite{chatterjee2025detectingweightedhiddencliques}, which adapts a model where the weights of the planted clique are drawn from a separate distribution, thus replacing the associated edges of the Erd\'os-R\'enyi graph rather than additively perturbing the model.

\subsection{Notation}

$\|\cdot\|_2$ and $\| \cdot\|$ are both used to denote the operator norm for the matrix. $\|\cdot\|_F$ is used as the Frobenius norm.   $\|X\|_{\psi_l}$ represents the sub-Gaussian ($l=2$) or sub-exponential ($l=1$) norm $$\|X\|_{\psi_l}=\sup_{t\ge1}t^{-l}(\mathbb E |X|^t)^{l/t}=\inf\{K>0 : \mathbb E \exp(|X|^l/K^l)\le 2 \}~.$$
We use $(\lambda_i(X), u_i(X))$ to be the top $i$-th normalized leading eigenpair of a matrix $X$. Further, we adopt the descending eigenvalue ordering $\lambda_1\ge\lambda_2\ge\cdots\ge\lambda_n$.

We use standard asymptotic notations $O(\cdot), \Theta(\cdot),$ and $ \Omega(\cdot)$. Further, we say $A\lesssim B$ if $A\le CB$ for some absolute constant $C$ and $A \ll B$ if $A/B=o(1)$.

For a set $A$, let dist$(A,b)=\min_{a\in A}|a-b|$. We use spec$(X)$ to represent the set of eigenvalues for a matrix $X.$ We use $\mathcal N(X)$ to represent the null space of a matrix $X.$

%% file: sections/model.tex
\section{Model Setup}

In this section, we formally introduce our model in \eqref{model} by first defining its components. We use $\odot$ as the Hadamard product for vectors and matrices: $(Y\odot Z)_{ij}=Y_{ij}Z_{ij}$, for $i,j\in [n]$.

\subsection{The Spikes}

We begin by considering the underlying sparse signals. We fix $r\in\mathbb{N}$ as the number of distinct signals. For $i\in [r],$ we denote the signal-to-noise ratios $\theta _i>0$ paired with the spike vectors $v_i=\tilde{v_i} \odot b_i\in\mathbb{R}^n$ where $\tilde{v_i}~\iidsim~ \tilde{\mathcal V}$ for some known sub-Gaussian prior distribution $\tilde{\mathcal V}$ in $\mathbb{R}^n$ and $b_i ~\iidsim~\mathcal B$ for all $i\in[r]$. Here, $\tilde{\mathcal V}$ represents the dense, sub-Gaussian prior, and $\mathcal B$ represents our method of sparsity sampling: 
\begin{assumption}[Centered, Unit Variance Sub-Gaussian Vectors]
    For $i\in [r]$, $\tilde v_i ~\iidsim ~ \tilde{\mathcal V}$ if $\mathbb E|(\tilde v_i)_j|^t=O(t)^{t/2}$ as $t\to\infty$ where $\|(\tilde v_i)_j\|_{\psi_2}\le K$ for all $j\in [n]$. Furthermore, $\mathbb E[(\tilde v_i)_j]=0$ and $\mathbb E [(\tilde v_i)_j^2]=1$ for all $i\in [r]$ and $j\in[n]$.
\end{assumption}

\begin{assumption}[Independent Sparsity]\label{def:sparse_1}
    For $i\in [r],$  when $b_i~\iidsim~ \mathcal B,$ we consider the natural Bernoulli sparsity applied to the vector, i.e., $(b_i)_j$ are i.i.d. $  \text{Ber}(p)$ for all $j\in [n]$.
\end{assumption}
Observe how the support size of the spike vectors is $np$ only in expectation; however, if $np\to \infty$ (as is the case under Assumption \ref{p_assumption}), we can condition on the size of the support, then use Lemma \ref{support_size_independent} to control its fluctuation around its mean. Notice this sparsity prior does allow for potential overlap in support across the spike vectors $v_i$ when $np$ is large enough, which is allowed by Assumption \ref{p_assumption} as there is no dependence on orthogonality between signal vectors in our main proofs.  

We can then consider the $r$ spike additions as one matrix perturbation 
$$\theta_1v_1v_1^T+\theta_2v_2v_2^T+\cdots+\theta_rv_rv_r^T=V\Theta V^T~,$$ where we have matrices 
\begin{equation}\label{eq:def_spikes}
    V=\begin{bmatrix}
        v_1 & v_2 & \cdots & v_r
    \end{bmatrix}\qquad \text{and }\qquad \Theta =\text{diag}(\theta _1,~ \theta _2,~ \dots,~ \theta _r)~,
\end{equation}
with $\theta_1>\theta_2>\cdots>\theta_r>0$. Our analysis extends trivially to both negative and repeated signal-to-noise ratios. For instance, applying the procedure to $-X$ would produce the same results, but would create significant eigenvalues on the left side of the bulk. The case of repeated $\theta_i$ values introduces complications as specific signal vector recovery becomes impossible; \cite{benaychgeorges2010eigenvalues} provides a template on how to handle this case, and we provide similar results corresponding to our model in Appendix \ref{equal_signals}.

\subsection{The Noise}

Next, we consider the noise. We start out with a dense, real Wigner matrix $W \in \mathbb R^{n\times n}$ satisfying the following assumption: 
\begin{assumption}[Sub-Gaussian Wigner Matrix]\label{wigner_subg_assumption}
    For all $\{i,j \}\in [n]^2,$ $W$ is symmetric with centered ($\mathbb E[W_{ii}]=\mathbb E [W_{ij}]=0$), finite variance ($\mathbb E[W_{ii}^2]=2$ and $\mathbb E[W_{ij}^2]=1$), sub-Gaussian entries. Specifically, consider the log-Laplace transform of the entries of $W$: $\Lambda_{i,j}(t):=\log \mathbb E\exp (tW_{i,j})$. Then, for all $\{i,j\}\in [n]^2$, $$\sup_{t\neq 0}\left\{\frac{\max\{\Lambda_{i,i},\Lambda_{i,j},  \}}{t^2} \right\}<\infty~.$$ Moreover, for any $r\ge1$, we assume $\max_{i,j}(\mathbb E(W_{ij}^2) )^{1/2}\le K_2r^{1/2}$ where $K_2\ge1$.
\end{assumption} Next, we introduce the noise mask matrix $A \in \mathbb R^{n\times n}$ satisfying the following assumption: 
\begin{assumption}[Bernoulli Noise Mask]\label{noise_mask_assumption}
    $A$ is a symmetric matrix $A_{ij}=A_{ji} \sim$ Ber$(q)$ for all $\{i,j\}\in[n]^2$.
\end{assumption}
Our noise matrix is thus the Hadamard product matrix $W\odot A$.

\subsection{The Model}

Therefore, we have the following real $n\times n$ matrix representing our \textit{doubly sparse spiked Wigner model}: 
\begin{align}\label{model}
     X=\frac{1}{np}V\Theta V^T+\frac{1}{\sqrt{nq}}W\odot A~.
\end{align}
The above normalizations give us $\|\frac{1}{np}v_iv_i^T\|_2\xrightarrow[]{\P}\Theta(1)$ and $\|\frac{1}{\sqrt{nq}}W\odot A\|_2\xrightarrow[]{\P}\Theta(1)$, thus ensuring the spike and noise have comparable spectral norms, which allows us to keep $\theta_i=\Theta(1)$.

\subsection{The Tasks}
We are interested in two tasks in the context of this paper. $\mathbb Q$ represents the null distribution, meaning $\Theta = 0$ when $X\sim \mathbb Q$, and $\mathbb P$ represents the planted distribution where $\Theta \not =0$ for $X\sim \mathbb P$.

\begin{definition}\label{distinguishdef}
    Given the planted $\mathbb P$ and null models $\mathbb Q$, we can \textit{distinguish}\footnote{Also referred to as Strong Detection in the literature.} if there exists a test statistic $\mathcal T$ such that $\mathbb P(\mathcal T(X)<\epsilon_0)+\mathbb Q(\mathcal T(X)\ge \epsilon_0)\to 0$ as $n\to \infty$ where $\epsilon_0$ is some established threshold parameter.
\end{definition}

\begin{definition}\label{weaklyrecoverdef}
    Given $X\sim \mathbb P$, we can \textit{weakly recover}\footnote{Other levels of recovery exist, such as Partial, Almost Exact, and Exact, but they depend on an atomic prior and are thus not general.} if there exist estimators $\mathcal M_i$ such that $$|\langle\mathcal M_i(X), v_i\rangle|/\| v_i\|_2>0$$ with probability tending to $1$ as $n\to \infty$.
\end{definition}

\paragraph{Remark. } In this work, we provide positive results for these tasks, i.e. specific threshold values for the BBP phenomenon above which the problem can be solved. For more information on obtaining negative results (conditions under which the problem cannot be solved regardless of the method employed), see Section \ref{beyondbbp}.


%% file: sections/main_results.tex
\section{Main Results}\label{sec:main}
\begin{assumption}[Sparsity Assumptions]\label{p_assumption}
    For the noise sparsity, assume $q\gg \frac{\log n}{n}$; more specifically, let $q:=\tau\frac{\log n}{n}$ such that $\tau:=\tau(n)\to\infty$ as $n\to\infty$. For the spike sparsity, assume $p\gg\frac{1}{n}$. 
    Thus, we have $np\to\infty$ and $nq\to\infty$. 
\end{assumption}

\begin{definition}[Probability Bound Function]
    For any $\gamma>0,$ let 
    \begin{equation}\label{eq:rate_func}
        \bm f_r(n):=\max\left\{r(r-1)\exp\left(-\frac{c_1}{4C_3C_2K^4}\frac{np}{(\log(np))^{2\gamma}}\right), ~ rn^{\max\{-2,-D\}}, ~\exp(-C_8nq)\right\}~,
    \end{equation}
    where $p$ and $q$ are functions of $n,$ $r$ represents the number of spikes, and $c_1, C_2, C_3, C_8, K$, and $D$ are positive constants defined in Theorems \ref{eigenvalue_theorem} and \ref{eigenvector_theorem}. Note $\bm f_r(n)\to0$ as $n\to \infty$ under Assumption \ref{p_assumption}.
\end{definition}


\begin{theorem}[Doubly Sparse BBP: Eigenvalues] \label{eigenvalue_theorem}
    Given Assumption \ref{p_assumption}, we have the following behavior for the extremal eigenvalues of $X$ in \eqref{model} for all fixed $1\le i\le r$: 
\begin{align*}
    \lambda_i(X) \xrightarrow{} 
    \begin{cases}
        \theta _i + \frac{1}{\theta _i} & \theta _i > 1 \\
        2 & \theta _i \le 1 ~\text{ or } i>r
    \end{cases}~,
\end{align*}
in probability. More specifically, for any $\theta_i>1$ and $\gamma>0$, there exist some constants $c_1, C_2, C_3, C_5, C_7, C_8$, and $D>0$ such that 
\begin{align*}
    \mathbb P \left( \left | \lambda_i(X)-\left(\theta_i+\frac{1}{\theta_i}\right) \right| \le C_7\theta_i\Lambda \right) \ge 1-\bm f_r(n)~,
\end{align*}
holds, where $\bm f_r(n)$ is defined by \eqref{eq:rate_func} and
$$\Lambda = \|\Theta\|_F\sqrt{C_5r^2\max\{(\log(np))^{-2\gamma}, ~ \tau^{-1/2}\}}~.$$
\end{theorem}



\begin{corollary}[Distinguishability]
    Under Assumption \ref{p_assumption}, given $\epsilon>0$, we can use the largest eigenvalue $\lambda_1(X)$ to distinguish between the planted $\mathbb P$ and null $\mathbb Q$ models \textbf{if $\theta_1>1$}: 
        \begin{itemize}
            \item If $\lambda_1(X)>2+\epsilon$, we are in the planted model;
            \item If $\lambda_1(X)<2+\epsilon,$ we are in the null model.
        \end{itemize}
\end{corollary}

\begin{proof}
    Recalling Definition \ref{distinguishdef}, we choose test statistic $\mathcal{T}=\lambda_1(X)$. By Lemma \ref{semicirc_law_sparse_lemma}, we have $\mathbb{Q}(\mathcal{T}(X)>2+\epsilon)=\mathbb Q(\lambda_1(\frac{1}{\sqrt{nq}}W\odot A)>2+\epsilon)\to0$ with high probability. By Theorem $\ref{eigenvalue_theorem}$, we have $\mathbb P(\mathcal T(X)<2+\epsilon)\to0$ with high probability. 
\end{proof}

\begin{corollary}
    Under Assumption \ref{p_assumption}, we can detect the presence of $s$ spikes where $s=\max\{i\in [r]:\theta_i>1\}$.
\end{corollary}

\begin{theorem}[Doubly Sparse BBP: Eigenvectors]\label{eigenvector_theorem}
Under Assumption \ref{p_assumption}, we have the following behavior for the unit-norm eigenvector $u_i(X)$ of the $i$-th largest eigenvalue of $X$ in \eqref{model}:  
\begin{align*}
    \langle u_i(X), \frac{1}{\sqrt{np}}v_j \rangle^2 \xrightarrow{\P} 
    \begin{cases}
        1-\frac{1}{\theta _i^2} & \theta _i > 1 \\
        0 & \theta _i \leq 1 \text{ or } i\neq j
    \end{cases}
\end{align*}
where $v_j$'s are defined by \eqref{eq:def_spikes}, for all $i,j\in [r]$. More specifically, for any $\theta_i>1$ and $\gamma>0$, there exist some positive constants $c_1, C_2, C_3, C_5, C_7, C_8$, and $D>0$ such that for $ i\in [r]$, 
\begin{align*}
    \mathbb P \left(  \left| \langle u_i(X), \frac{1}{\sqrt{np}}v_i \rangle^2-\Big(1-\frac{1}{\theta_i^2}\Big)\right| \lesssim \max\{(\log(np))^{-\gamma}, ~ \tau^{-1/4}\} \right) \ge 1-\bm f_r(n)~,
\end{align*}
and, for $j\neq i\in [r]$, 
\begin{align*}
    \mathbb P \left( \langle u_i(X), \frac{1}{\sqrt{np}}v_j \rangle^2 \lesssim  \Lambda^2 \right) \ge 1-\bm f_r(n)~,
\end{align*}
hold where $\bm f_r(n)$ is defined by \eqref{eq:rate_func} and
$$\Lambda = \|\Theta\|_F\sqrt{C_5r^2\max\{(\log(np))^{-2\gamma}, ~ \tau^{-1/2}\}}~.$$
\end{theorem}


\begin{corollary}[Weak Recovery]
    Under Assumption \ref{p_assumption}, we can use the $i$-th largest eigenvector as a candidate to weakly recover the spike vector $v_i$ if $\theta_i>1$ for $i\in [r]$.
\end{corollary}

\begin{proof}
    Recalling Definition \ref{weaklyrecoverdef}, we choose  estimators $\mathcal{M}_i=v_i(X)$ with unit norm achieved by performing the power method on the matrix $X$. By Theorem \ref{eigenvector_theorem}, we have $\langle \mathcal{M}_i(X),\frac{1}{\sqrt{np}}v_i \rangle^2 \to 1-\frac{1}{\theta_i^2}$ with high probability.
\end{proof}

%% file: sections/experiments.tex
\section{Numerical Illustration}\label{experiments}
In this section, we present experiments for various regimes considered in our theorems. For the first three figures, Figures~\ref{fig:regime_1}—\ref{fig:regime_3}, we examine the case of one spike ($r=1$). In all of these experiments, (a) demonstrates how the $\theta_1=3$ spike generates the top eigenvalue $\lambda_1(X)$, represented by the blue \textcolor{blue}{$\bm{\times}$} symbol, which lies outside the semicircular bulk and aligns with the predicted value $\theta_1+\frac{1}{\theta_1}$, represented by the orange circle \textcolor{orange}{$\bm{\circ}$} symbol. Furthermore, (b) compares eigenvector alignments between the estimated vector $u_1(X)$ and the true signal vector $v_1$ with the predicted value, $(1-\frac{1}{\theta_1^2})$, derived from Theorem~\ref{eigenvector_theorem}, as we increase the strength of the signal $\theta_1$. Notably, in the latter two figures, the error bars exhibit wider ranges as the concentration of the sparse vector $v_1$ is not as robust as the dense case. As seen in the fluctuation results of Theorem \ref{eigenvector_theorem}, this demonstrates how the fluctuations are significantly influenced by $p$. 

Finally, we present Figure~\ref{fig:two_spikes} depicting a case with multiple spikes ($r=2$). In this case, we consider double sparsity for spikes and Wigner matrix: $p=\frac{3(\log n)^2}{2n},$ and $q=\frac{5(\log n)^3}{n}$. Part (a) is similar to Figures~\ref{fig:regime_1}—\ref{fig:regime_3} but with two outliers. For (b), we present the entries of the top three eigenvectors with unit norm, $u_1(X),u_2(X),u_3(X)$, in the spectrum of (a). In this case, the top outlier eigenvectors $u_1(X)$ and $u_2(X)$ are sparse as they are aligned with the sparse signals $v_1,v_2$ based on Theorem \ref{eigenvector_theorem}; the third eigenvector is from the bulk and is delocalized.

\begin{figure}[ht]
\begin{minipage}{.6\textwidth}
    \centering
    \includegraphics[height=5cm]{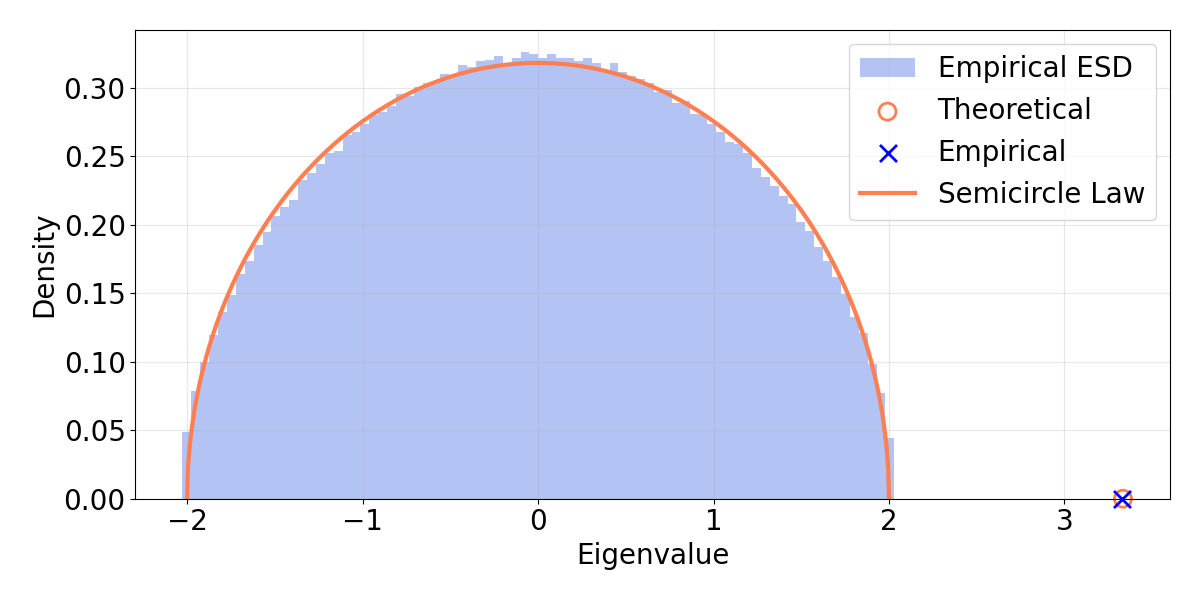}\\
    {\small (a)}
\end{minipage}%
\begin{minipage}{.39\textwidth}
    \centering
    \includegraphics[height=5cm]{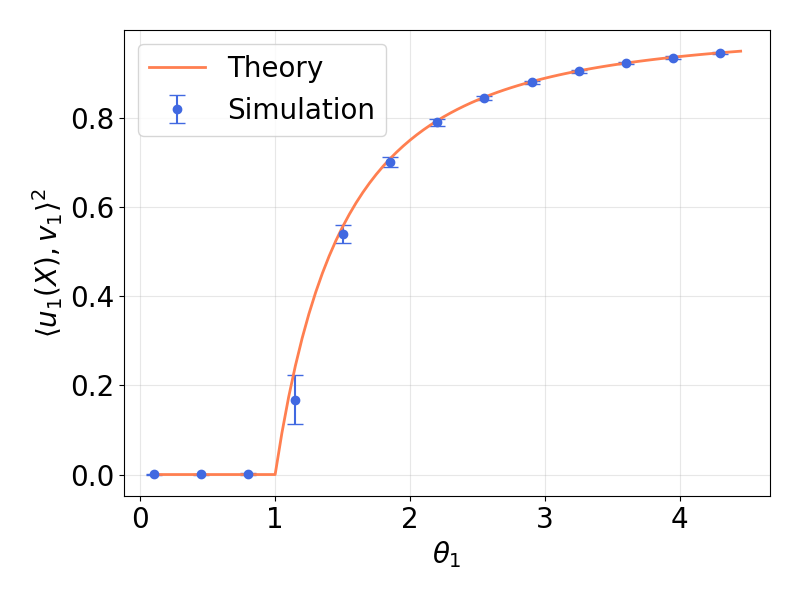}\\
    {\small (b)}
\end{minipage}
\caption{The dense spike ($p=1/2$) and sparse noise ($q=\frac{(\log n)^2}{n}$) case, when $r=1$ and $n=15000$.
(a) The eigenvalue distribution of $X$ compared to the theoretical prediction of the limiting eigenvalue distribution and the location of the outlier eigenvalue. Here $\theta_1 = 3$.
(b) The alignment of top eigenvectors with the spike vector $\langle u_1(X),v_1 \rangle^2$ for various signal-to-noise ratios $\theta_1$. Each experiment is run 30 times to obtain an average.}\label{fig:regime_1}
\end{figure}

\begin{figure}[ht]
\begin{minipage}{.6\textwidth} 
    \centering
    \includegraphics[height=5cm]{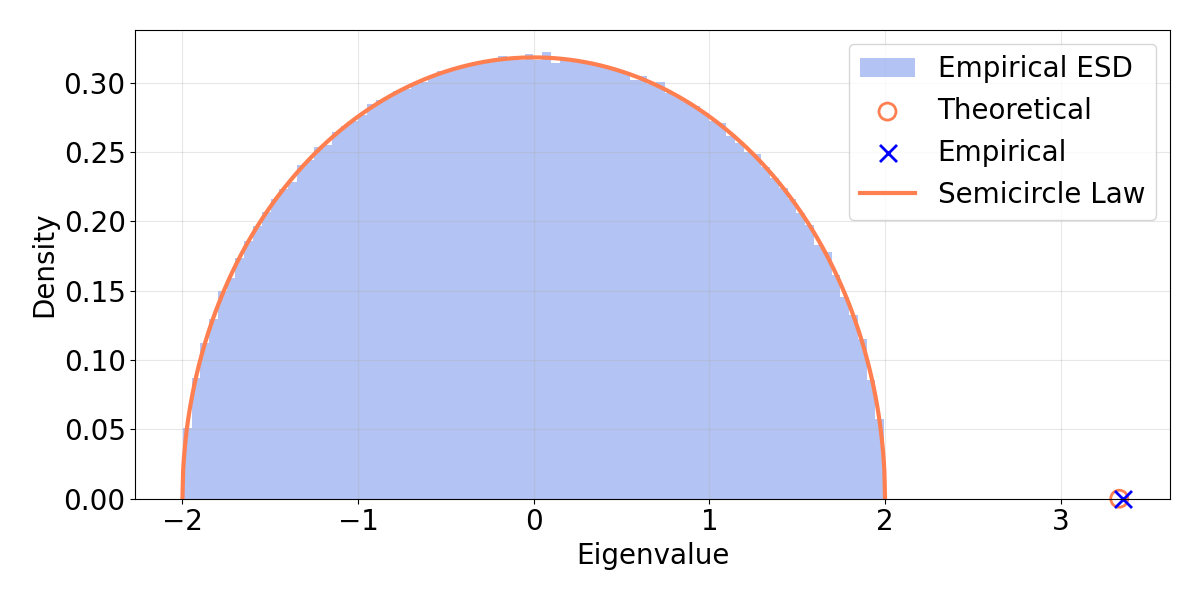}\\
    (a)
\end{minipage}%
\begin{minipage}{.39\textwidth}
    \centering
    \includegraphics[height=5cm]{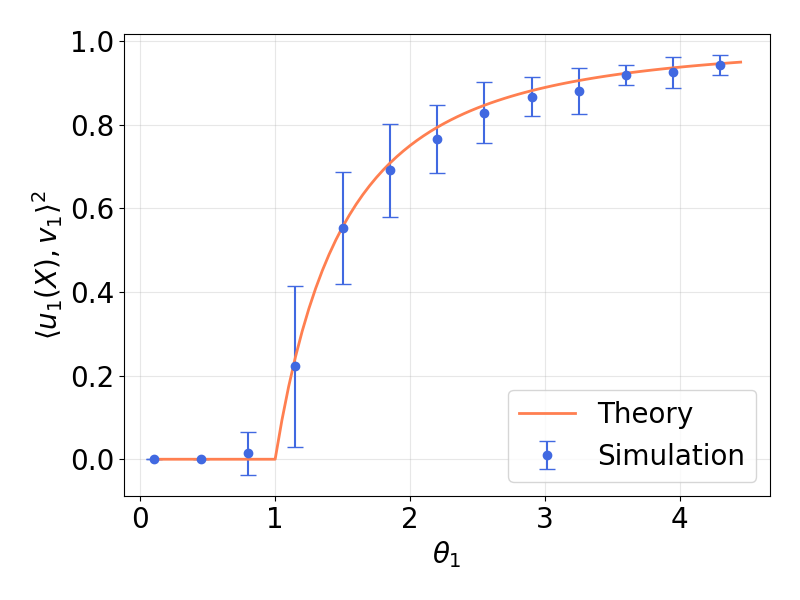}\\
    (b)
\end{minipage}
\caption{The sparse spike ($p=\frac{(\log n)^2}{n}$) and moderately dense noise ($q=1/2$) case, when $r=1$ and $n=15000$.
(a) The eigenvalue distribution of $X$ compared to the theoretical prediction of the limiting eigenvalue distribution and the location of the outlier eigenvalue. Here $\theta_1 = 3$.
(b) The alignment of top eigenvectors with the spike vector $\langle u_1(X),v_1 \rangle^2$ for various signal-to-noise ratios $\theta_1$. We take an average over 30 runs.}\label{fig:regime_2}
\end{figure}

\begin{figure}[ht]
\begin{minipage}{.6\textwidth}
    \centering
    \includegraphics[height=5cm]{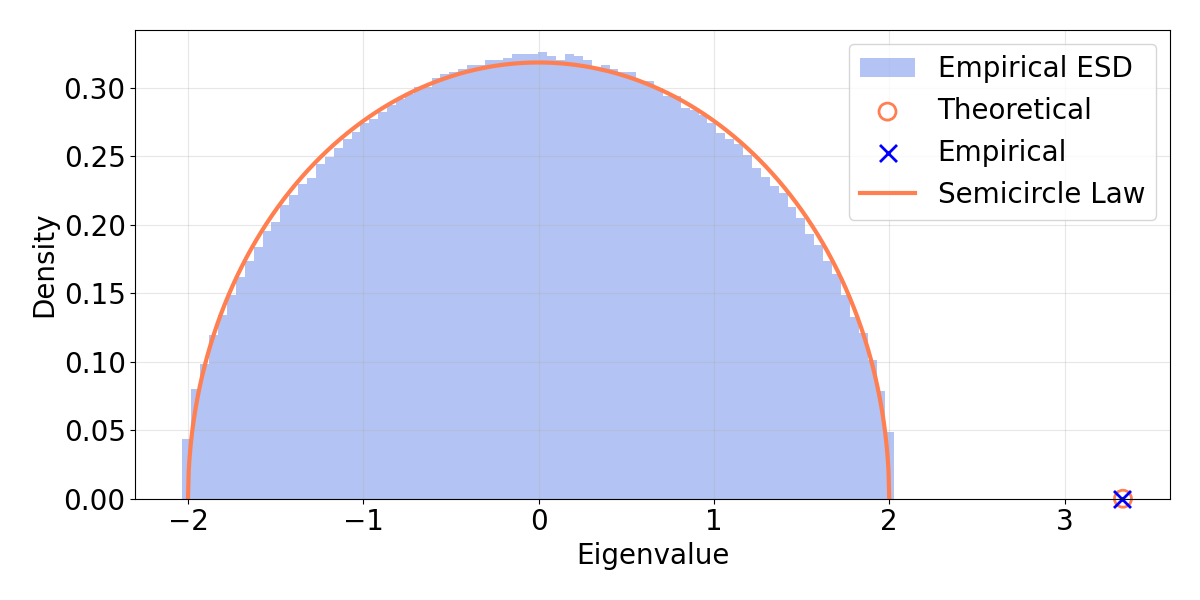}\\
    (a)
\end{minipage}%
\begin{minipage}{.39\textwidth}
    \centering
    \includegraphics[height=5cm]{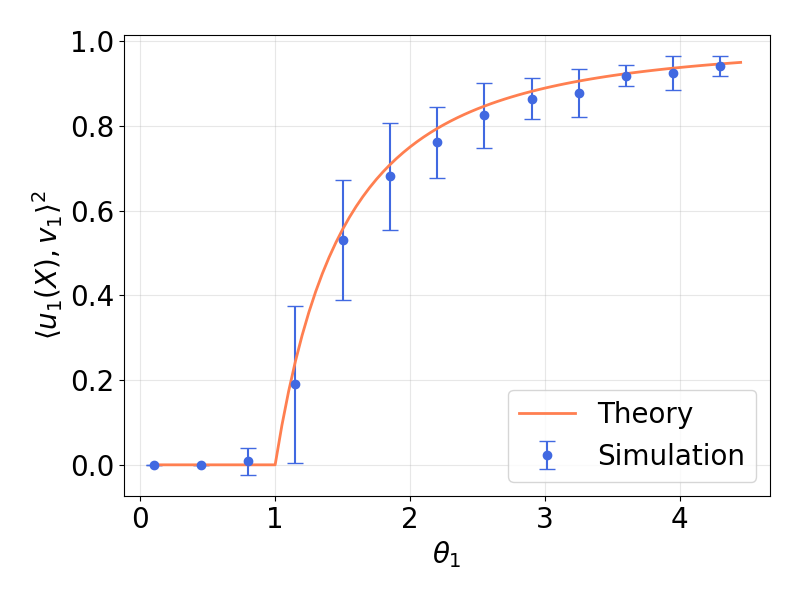}\\
    (b)
\end{minipage}
\caption{The sparse spike and sparse noise ($p=q=\frac{(\log n)^2}{n}$) case, when $r=1$ and $n=15000$.
(a) The eigenvalue distribution of $X$ compared to the theoretical prediction of the limiting eigenvalue distribution and the location of the outlier eigenvalue. Here $\theta_1 = 3$.
(b) The alignment of top eigenvectors with the spike vector $\langle u_1(X),v_1 \rangle^2$ for various signal-to-noise ratios $\theta_1$. We take an average over 30 runs. }\label{fig:regime_3}
\end{figure}

\begin{figure}[ht]\label{multispike}
\begin{minipage}{.75\textwidth}
    \centering
\includegraphics[height=5cm]{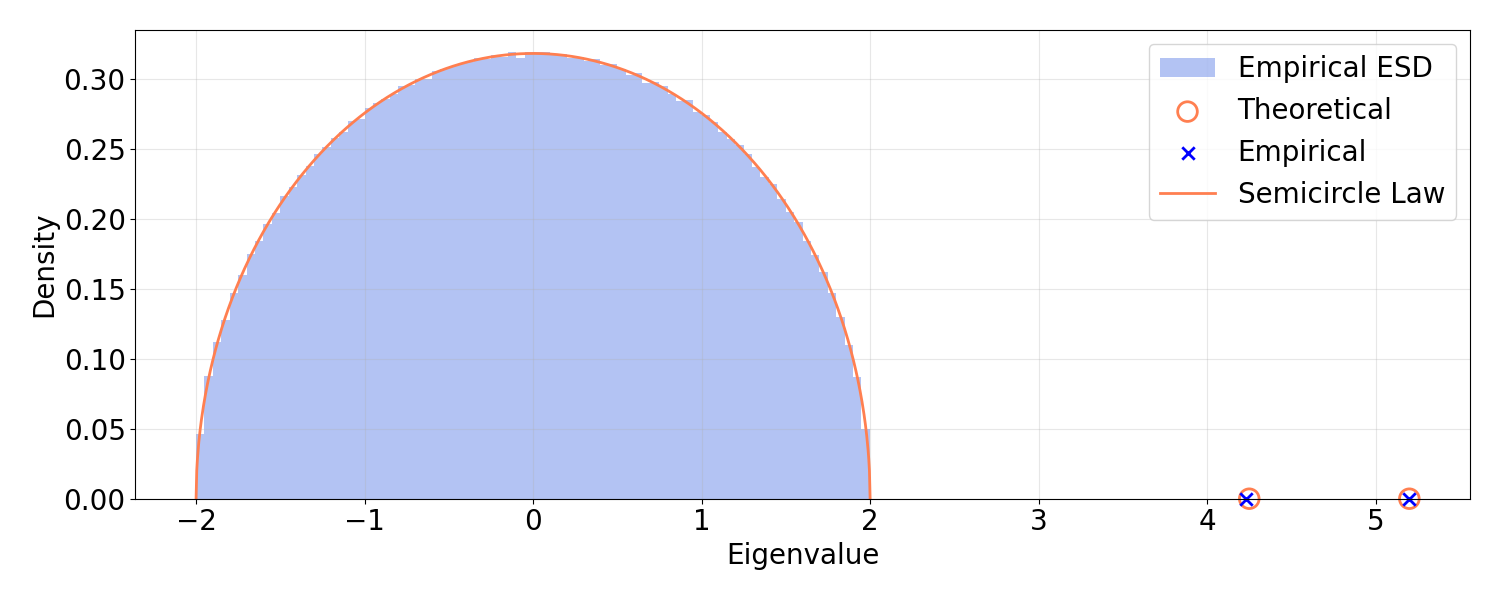}\\
    (a)
\end{minipage}%
\begin{minipage}{.15\textwidth}
    \centering
\includegraphics[height=5.2cm]{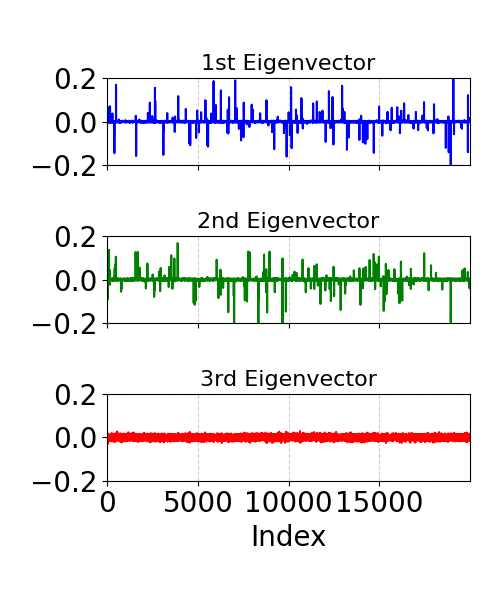}\\
    (b)
\end{minipage}
\caption{\small The sparse spike and sparse noise ($p=\frac{3(\log n)^2}{2n}, q=\frac{5(\log n)^3}{n}$) case, when $r=2$ and $n=20000$.
(a) The eigenvalue distribution of $X$ compared to the theoretical prediction of the limiting eigenvalue distribution and the location of the outlier eigenvalue. Here $\theta_1 = 5, \theta_2=4$.
(b) The plots of the entries of top eigenvectors with unit norm, $u_1(X),u_2(X),u_3(X)$ from (a). Notice that the first two eigenvectors corresponding to spikes are localized and sparse because of the sparse signals from Theorem~\ref{eigenvector_theorem}; while the third eigenvector from the bulk is delocalized.}\label{fig:two_spikes}
\end{figure}

%% file: sections/future_directions.tex
\section{Future Directions}\label{sec:future}

\subsection{Extensions Within BBP}

An avenue of exploration directly related to this work involves a result produced in \cite{péché2005largesteigenvaluesmallrank} which extends the BBP to an $n$-dependent number of spikes $r$. Within the context of this work, we see that the reliance on the finite nature of the $r$ by $r$ matrix (specifically with regards to continuity and convergence properties) serves as the main obstacle to adding an asymptotically growing number of spikes. However, these technical limitations may be removable.


Additionally, one may consider expanding the regime of sparsity for $q$ to the critical and sub-critical regimes. As discussed in Section \ref{doub_sparse_lit_review}, shifting to the assumption of bounded entries in the noise matrix allows for critical sparsity ($q \gtrsim \frac{\log n}{n})$ for our result of no spectral outliers, but we would still need a local law result in the style of Lemma \ref{lemma_2} for critical sparsity. Further, we believe no BBP style result could be produced in the sub-optimal regime as there would be no way to reliably distinguish between outliers generated by the spikes or spurious outliers from the noise. Thus, for $q\lesssim \frac{\log n}{n}$, a non-spectral approach would be required. 

\subsection{Sparse PCA: Beyond Spectral Methods}\label{beyondbbp}

\paragraph{Contiguity, AMP, and Statistical Physics. } While PCA provides a simple solution to both problems in the case of $\theta >1$, the spectral threshold does not always align with the \textit{information-theoretical threshold}. This misalignment stems from PCA assuming very little about the spike; indeed, information can be extracted from knowledge of the spike prior $\mathcal V$ in the Bayesian setting. For Gaussian, dense noise ($q=1$), a series of works showed that, for certain spike priors and low enough $\theta$, the planted and null models are contiguous, meaning there is no statistic that can distinguish between the spiked and null models. \cite{montanari2014limitationspectralmethodsgaussian} and \cite{johnstone2018testinghighdimensionalspikedmodels} proved this for deterministic and spherical priors, respectively, and  \cite{perry2016optimalitysuboptimalitypcaspiked} followed with results for certain dense spike priors (spherical, Rademacher, any ``sufficiently'' sub-Gaussian). However, the latter showed PCA is sub-optimal for non-Gaussian Wigner matrices, leaving the door open for methods that improve on the BBP threshold. Further, their contiguity results for Gaussian Wigner extend only to the above-mentioned spike priors, leaving room for results on Gaussian Wigner models where the spike is drawn from a different prior (such as a sparse one). Various techniques have arisen to construct methods that exploit the signal's known structure. Born from a statistical physics setting and first applied to problems in compressed sensing \cite{Bayati_2011, Donoho_2009}, Approximate Message Passing (AMP) algorithms are optimal among all polynomial-time algorithms when applied to the spiked Wigner model \cite{Bayati_2011, montanari2014limitationspectralmethodsgaussian, deshpande2014informationtheoreticallyoptimalsparsepca}. Their methods work particularly well when exploring information-theoretic thresholds due to their performance being characterized by a simple deterministic recursion called ``state evolution'' in the high-dimensional limit. When considering a sparse spike, employing such methods allows for solving the problem below the threshold $\theta <1$. Using standard statistical physics heuristic based methods, \cite{Lesieur2015} employed these AMP techniques to analyze the minimum mean squared error and provide evidence of the threshold falling below the spectral one when the spike is sufficiently sparse. While these results serve as fortified conjectures, they prompted further work by Perry et al. \cite{perry2016optimalitysuboptimalitypcaspiked} and Banks et al. \cite{banks2017informationtheoreticboundsphasetransitions} to explore the values of $p$ and $\theta$ for which the problem can be solved in the Gaussian Wigner noise models. Both papers used careful conditioning on the second moment of the likelihood ratio to discover similar thresholds for the problems of distinguishability and recovery. For $p\le p_*\approx0.6$, the threshold for $\theta$ falls below the BBP threshold. For constant $p$, a wide gap emerged between the solvability thresholds; however, in the limit $p\to0$, their gap shrinks to $\sqrt{2}$. This limit is taken outside of the context of their work, as they consider $p$ to be a constant throughout their proofs. For simplicity, the two works used a sparse Rademacher spike prior in which $\tilde v~\iidsim  ~\text{Unif}(\{\pm 1\})$ and $b~\iidsim~\mathcal B_1$. However, they claim the Rademacher structure can naturally be extended in many of their results to any spike prior with bounded support. In fact, Alaoui et al. \cite{El_Alaoui_2020} adapted the above mentioned statistical physics methods of free energy to provide rigorous results and a generalized threshold stylized to the choice of any bounded prior, thus closing the $\sqrt{2}$ gap for the sparse Rademacher spike prior. Moreover, all of the above results consider the prior to be controlled by an $n$-independent, constant $p$, whereas our work allows for sparsity $p$ to scale with $n$. 
 
\paragraph{Thresholding. } An avenue of study tangential to this current work, though pertinent to the understanding of recovering the sparse spike, involves a non-spectral approach of thresholding the entries of the spiked matrix to construct a polynomial-time algorithm for recovery. Indeed, in upcoming work, we have proved a series of results involving thresholding for our doubly sparse spiked matrix model. In the context of this paper, we provide a necessary literature review of this method's impact on the problem of recovery. Thresholding has proven to be an effective tactic for providing polynomial-time algorithms to solve a spiked model in the sparse setting. In the Wishart model where $1>\theta =\Theta(1) $ and $v$ is sparse in a suitable basis (such as the wavelet domain), Johnstone first proposed a simple two-step diagonal thresholding algorithm, which begins with an initial pre-processing step to reduce the space, followed by ordinary PCA \cite{johnstone2009consistency}. Narrowing the spike prior to Sparse Rademacher, Wainwright and Amini improved Johnstone's sparsity assumption of $p\lesssim\frac{1}{n}$  to $p\lesssim \frac{1}{\sqrt{n\log n}}$ \cite{Amini_2009}. Shifting to a thresholding method on the covariance, the sparsity condition relaxes to as large as $p\lesssim \frac{1}{\sqrt{n}}$ \cite{Krauthgamer_2015, deshpande2016sparsepcacovariancethresholding}. Considering the possibility of exhaustively searching over all $\binom{n}{np}$ possible support sizes (and thus ignoring any computational consideration), recovery can be achieved for $p\lesssim \frac{1}{\log n}$ \cite{paul2012augmentedsparseprincipalcomponent, vu2012minimaxratesestimationsparse, Cai_2013}, but the covariance thresholding still holds as the best known polynomial-time algorithm. With knowledge that the problem can be solved in an information-theoretical sense for constant values of $p$ (and an appropriately chosen $\theta$ depending on the value of $p$), the question naturally arises: why are there no known polynomial-time algorithms to match these positivity results? Reducing the problem from the \textit{planted clique} problem would seem to support the notion that the problem exhibits a ``possible but hard'' regime when $\frac{1}{\sqrt{n}}\ll p\ll 1$. Under the spike prior assumption of bounded entries with $\|v \|\to 1$ and fixed spike support $np$, Ding et al. explore this very question and provide an algorithm that reveals a precise tradeoff between its sub-exponential runtime and sparsity $p$ \cite{ding2022subexponentialtimealgorithmssparsepca}. Further, while all the previously mentioned works only consider the Wishart model, the work of Ding et al. extends these thresholding results to both models and suggests that the thresholding outcomes as a whole are equivalent across both models\footnote{Further, as discussed immediately following the introduction of Theorem \ref{general_bbp}, there is evidence that the two spiked models belong to the same equivalence class.}. 

We have developed two different thresholding algorithms for our model, which we intend to include in a companion paper following this work. These two algorithms along with our BBP result cover different but incomplete regimes of sparsity and $\theta$ values for the model in \eqref{model}.


%% file: sections/preliminary.tex
\section{Preliminary Lemmas}\label{prelim_lemmas}

We begin by establishing necessary lemmas regarding the behavior of the sparse noise matrix, which we use in our main BBP proofs. We first define $$m(z):=\frac{1}{2\pi}\int_{-2}^{2} \frac{\sqrt{4-x}}{x-z}dx=\frac{-z+\sqrt{z^2-4}}{2}, \hspace{3mm} z\in\mathbb C\backslash[-2,2]$$ to be the Stieltjes transform of the semicircular distribution.

\subsection{No Spectral Outliers for Sparse Wigner Matrix}
This first lemma shows that, when $q\gg \frac{\log(n)}{n}$, the eigenvalues of the sparse noise stay within $(-2-\epsilon,2+\epsilon)$ for $\epsilon>0$ with high probability. This ensures that, if we observe an eigenvalue outside of this interval, it is a result of the spike (with high probability). 

\begin{lemma}{Let $A$ be a symmetric matrix satisfying Assumption \ref{noise_mask_assumption}, and $W$ be a Wigner matrix satisfying Assumption \ref{wigner_subg_assumption}. Then, for any $\epsilon>0$ and $q\gg \frac{\log(n)}{n}$, we have $$\mathbb P \left( \lambda_1(\frac{1}{\sqrt{nq}}W\odot A)>2+\epsilon\right) \le \exp(-C_8nq)~,$$ where $C_8$ is the constant \eqref{c_8_constant} defined within the proof.}\label{semicirc_law_sparse_lemma}
\end{lemma}

\begin{proof}
    From \cite[Theorem 1.4]{augeri2023largedeviationslargesteigenvalue}, the condition $\frac{\log n}{n} \ll q\ll 1$ implies 
    \begin{align*}
        \mathbb P\Big(\lambda_1(\frac{1}{\sqrt{nq}}W\odot A)>2+\epsilon\Big)\leq\exp(-I(2+\epsilon)nq)~, 
    \end{align*}
    where $I(\lambda)$ is defined to be 
    \begin{align*}
        I(\lambda)=\inf\{\frac{r^2}{2\alpha }+h_L(s) : r-m(\lambda)s=\lambda, r\geq0, s\geq 1 \}~,
    \end{align*}
    with 
    \begin{align*}
        h_L(s) = \sup_{\xi\in \mathbb R}\{\xi s-L(\xi)+1 \}~, \hspace{.2cm} L(\xi)=\mathbb E[\exp(\xi W_{1,2}^2)]~, \hspace{.2cm} \text{and} \hspace{.2cm} \alpha := \lim_{\xi\to \infty}\frac{2\log\mathbb E[\exp\xi W_{i,j}^2]}{\xi^2}<\infty~.
    \end{align*}
    To show $\frac{1}{\sqrt{nq}}W\odot A$ has no outliers in the semicircular distribution defined on $(-2-\epsilon,2+\epsilon)$, it suffices to show $I(2+\epsilon)>0$, seeing as $nq\to\infty$ by the condition $q\gg \frac{\log n}{n}$. The positive nature of the rate function $I(2+\epsilon)>0$ is noted in their work as Remark 1.5 (vi); however, we are concerned with finding an explicit lower bound for $I(2+\epsilon)$.
    
    We begin by noticing the square of the sub-Gaussian entries $W_{1,2}$ of $W$ are sub-exponential with $\|W_{1,2}^2\|_{\psi_1}=\sup_{p\ge1}p^{-1}(\mathbb E |W_{1,2}^2|^p)^{1/p}=\inf\{K>0 : \mathbb E \exp(W_{1,2}^2/K)\le 2 \}$ (see \cite{vershynin2018high}). Without loss of generality, we assume $\|W_{1,2}^2\|_{\psi_1}=1$. Thus, $\mathbb E \exp(W_{1,2}^2)\le 2$.
    We see 
    \begin{align*}
        h_L(s) = \sup_{\xi}\{\xi s-\mathbb E[\exp(\xi W_{1,2}^2)]+1 \} \ge s-\mathbb E[\exp( W_{1,2}^2)]+1\ge s-2+1=s-1 ~.
    \end{align*}
    where we chose $\xi=\|W_{1,2}^2\|_{\psi_1}=1$.
    Thus, we have
    \begin{align*}
        I(\lambda)\ge\inf_{s \geq 1} \left\{ \frac{\lambda^2+2\lambda m(\lambda)s+m(\lambda)^2s^2}{2\alpha} +s-1\right \}:=\inf_{s \geq 1} \{ g(s)\}~.
    \end{align*}
    Taking the first and second derivatives with respect to $s$, we find
    \begin{align*}
        g'(s)=\frac{\lambda m(\lambda)}{\alpha}+\frac{m(\lambda)^2s}{\alpha}+1~ \hspace{.2cm} \text{and} \hspace{.2cm} g''(s)=\frac{m(\lambda)^2}{\alpha}~.
    \end{align*}
    As the second derivative is positive for all values $s\geq1$, we can achieve the infimum as a minimum and are concerned with finding the minimum of this convex function. We see the minimum of $g(s)$ occurs at the zero $$s_0=-\lambda/m(\lambda) -\alpha/(m(\lambda))^2~.$$ When $\alpha\ge1,$ we see $s_0\le 1$, so the infimum occurs at $s=1$. When $\alpha<1$, the infimum is achieved at $s_0>1$. We thus factor $$g(s)=\frac{\lambda^2+2\lambda m(\lambda)s+m(\lambda)^2s^2}{2\alpha}+s-1=\frac{m(\lambda)^2(s+\frac{\lambda}{m(\lambda)})^2}{2\alpha}+s-1~,$$ and plug in our computed root $s_0$ to find: $$\inf_{s \geq 1} \{ g(s)\}=\begin{cases} \frac{(m(\lambda)+\lambda)^2}{2\alpha}, & \alpha\ge 1 \\ -\frac{\lambda}{m(\lambda)}-\frac{\alpha}{2m(\lambda)^2}-1, & \alpha <1
    \end{cases}~.$$ 
    To find an explicit constant, we now plug in $\lambda = 2+\epsilon$ and compute an lower bound for this infimum. For the latter function, one can inspect the function to see the infimum is lower bounded by $1$ when $\alpha\le1$. When $\alpha> 1,$ however, the infimum dips beneath the value of $1$. Observe how $$\frac{(m(2+\epsilon)+2+\epsilon)^2}{2\alpha}=\frac{(2+\epsilon)^2+(2+\epsilon)\sqrt{\epsilon^2+4\epsilon}+\epsilon^2+4\epsilon}{8\alpha}\ge\frac{1}{2\alpha}~.$$ Therefore, we let \begin{align}
        C_8:=\min\left\{\frac{1}{2\alpha},1\right\}~,\label{c_8_constant}
    \end{align}
    and reach our desired conclusion. 

    The result of \cite{augeri2023largedeviationslargesteigenvalue}, and thus the above analysis, does not hold for constant $q$ values; however, when $q$ is constant and does not depend on $n$, we return to a classical Wigner matrix with entries independent of the dimension. Thus, we can use Theorem 2.3.24 (Strong Bai-Yin Theorem) in \cite{tao2023topics} to achieve almost sure convergence for constant $q$. The theorem requires the entries of matrix to be centered with variance 1 and a finite fourth moment. Investigating the fourth moment of our matrix, we see $\mathbb E [(\frac{1}{\sqrt{q}}W\odot A)_{ij}^4]=\frac{1}{q^2}\mathbb E[W_{ij}^4]\mathbb E[A_{ij}^4]=\frac{1}{q}\mathbb E[W_{ij}^4]$. Seeing as $q$ is a constant and $W_{ij}$ is a sub-Gaussian random variable (and therefore bounded in its fourth moment by Assumption \ref{wigner_subg_assumption}), we see the conditions are satisfied. Therefore, when $q=\Theta(1)$, our result holds with probability $1.$
\end{proof}


\subsection{Local Law for Sparse Wigner Matrix}
Next, we prove a lemma which allows us to exploit an isotropic local law to show that the elements of the resolvent matrix $R_n(z)=(\frac{1}{\sqrt{nq}}W\odot A-zI)^{-1}$ are well approximated by the Stieltjes transform $m(z)$ outside the semicircular bulk. 
\begin{lemma}\label{lemma_2}
    Consider the resolvent $R_n(z)=(\frac{1}{\sqrt{nq}}W \odot A-zI)^{-1}$. For any $D>0$, $\epsilon>0$, and $\delta>0$, there exist a constant $C_1>0$ and $n_0=n(\delta, D, K_2, \epsilon)$ such that $q\geq C_1\frac{\log n}{\delta^4n}$ implies 
    $$ \mathbb P\left(\max_{1\leq i \leq n}|R_{n}(z)_{ii}-m(z)|\leq \bar C_1\delta\right) \geq 1-n^{-D}$$ for $z \not\in (-2-\epsilon, 2+\epsilon)$ and $n\ge n_0$ where $\bar C_1$  is a large enough constant depending on $z$ alone. 
\end{lemma}

\begin{proof}
    Seeing as $W$ has sub-Gaussian entries, we use \cite[Theorem 12]{sparsehanson} to choose a $D>0, \epsilon_0>0,$ and $\delta:=\delta(n)\to0.$ Then, there exist $n_0=n(D,\delta,K_2,\epsilon_0)$ defined in their work and constant $C_1=K_2^4*8^2*48^2*12^6\cdot\exp(2(D+7)+10+4/e)\cdot\max\{1,(\frac{5}{3}\delta)^4\}=K_2^4*8^2*48^2*12^6\exp(2(D+7)+10+4/e)$ for large enough $n$ (since $\delta\to 0$) such that $q\geq C_1\frac{\log n}{\delta^4n}$ implies 
    $$ \mathbb P(\max_{1\leq i \leq n}|R_{n}(z)_{ii}-m(z)|\leq \delta) \geq 1-n^{-D}~, \text{ for } z = E+i\eta \in S_0= \{E + i\eta : |E|\leq 10,n^{-1+\epsilon_0 } \leq \eta \leq 1 \}~,$$ where\footnote{In the context of their paper, we note $\beta = 1/2$ as $W$ has sub-Gaussian (2-subexponential) entries, giving us the lower bound for $\eta$ in $S_0$. } $\epsilon_0\in(0,1)$ for $n\ge n_0$.
    
    By our assumption, $z$ is a real number outside of $(-2-\epsilon,2+\epsilon)$. While $S_0$ is restricted to values with real parts less than $10$, this choice is noted to be arbitrary in \cite{sparsehanson} and can be changed without significantly altering the asymptotic result of the theorem (a change of constant in the term $C_1$ influencing sparsity $q$). However, we must take into consideration that the set $S_0$ does not cover values on the real line with an imaginary part equal to $0$. Using an estimation technique similar to \cite[Theorem 10.3]{knowleslectures}, we fix $z=E\not\in(-2-\epsilon,2+\epsilon)$, define $\kappa := \left||E|-2\right|>\epsilon$, i.e. the distance to the spectral edge, and consider $z_0 =E+i\eta_0$ where $\eta_0=\kappa^{2}\cdot \tilde{C}n^{-1+\epsilon_0}\in[n^{-1+\epsilon_0},1]$ for some suitably chosen constant $\tilde{C}$ to account for $\kappa^2$. Noting $z_0$ now satisfies the conditions of \cite[Theorem 12]{sparsehanson}, we claim it is enough to show, for an arbitrary unit-norm vector $w$,  
    \begin{align} \label{m1}
        |m(z)-m(z_0)|\leq \tilde{C}n^{-1+\epsilon_0} \text{ and }
    \end{align}
    \begin{align}\label{m2}
        |w^TR(z)w-w^TR(z_0)w|\leq \tilde C n^{-1+\epsilon_0}+\tilde C^2\kappa n^{-2+2\epsilon_0}~,
    \end{align}
    seeing as this would imply, with probability $1-n^{-D}$, that
    \begin{align*}
        \max_{1\leq i \leq n}|R_{ii}(z)-m(z)| =~& \max_{1\leq i \leq n}|R_{ii}(z)-R_{ii}(z_0)+R_{ii}(z_0)-m(z)+m(z_0)-m(z_0)| \\
        \leq~& \max_{1\leq i \leq n}|R_{ii}(z_0)-m(z_0)|+\max_{1\leq i \leq n}|e_{i}^TR(z)e_i-e_i^TR(z_0)e_i|+|m(z)-m(z_0)| \\
        \leq~& \delta +\tilde C n^{-1+\epsilon_0}+\tilde C^2\kappa n^{-2+2\epsilon_0}+\tilde{C}n^{-1+\epsilon_0} \\
        \leq~& \bar C_1\delta~,
    \end{align*}
    for some suitably chosen $\bar C_1$ large enough\footnote{Notice we assume $\delta$ is the dominating term as its presence in the denominator of the assumption on $q$ implies $\delta\ge (\frac{\log n}{n})^{-1/4}\ge n^{-1+\epsilon_0}$}.
    Beginning with (\ref{m1}), we use the Mean Value Theorem to see
    \begin{align*}
        |m(z)-m(z_0)|\leq|\max_sm'(s)|\cdot|z-z_0|=\max_s|\int\frac{1}{(s-x)^2}d\mu_{\rm sc}(x)|\cdot\eta_0
        \leq \frac{1}{\kappa^2}\cdot \kappa^{2}\cdot \tilde{C}n^{-1+\epsilon_0}=\tilde{C}n^{-1+\epsilon_0}~,
    \end{align*}
    where we used $|s-x|\geq\kappa$. Next, we recall the resolvent is a symmetric matrix as it is the inverse of a symmetric matrix. Thus, we can consider its spectral decomposition $$R_n(z_0)=\sum_{i=1}^n\frac{\mu_i\mu_i^T}{\rho_i-z_0}~,$$ where we used $(\rho_i,\mu_i)$ to represent the normalized eigenpairs of the sparse noise matrix $\frac{1}{\sqrt{nq}}W\odot A$. Thus, we move to (\ref{m2}) and observe 
    \begin{align*}
        |\Im w^TR(z_0)w|
        = \left|\Im \sum_{i=1}^n\frac{1}{\rho_i-z_0}(w^T\mu_i)^2\right| 
        =~& \left|\Im \sum_{i=1}^n\frac{\rho_i-E+i\eta_0}{(\rho_i-E-i\eta_0)(\rho_i-E+i\eta_0)}(w^T\mu_i)^2\right| \\
        =~& \left| \sum_{i=1}^n\frac{\eta_0}{(E-\rho_i)^2+\eta_o^2}(w^T\mu_i)^2\right| \\
        \leq~& \sum_{i=1}^n\frac{\eta_0}{\kappa^2}|w^T\mu_i|^2 \\
        \le~& \frac{\eta_0}{\kappa^2} \\
        =~&\tilde{C}n^{-1+\epsilon_0} ~,
    \end{align*}
    where we used Cauchy Schwarz between unit vector $w$ and normalized eigenvector $\mu_i$ $\forall i\in[n]$. 
    Next, we see
    \begin{align*}
        |\Re w^TR(z)w-\Re w^TR(z_0)w|=~& \left| \sum_{i=1}^{n}\frac{\rho_i-E}{(\rho_i-E)^2}|w^T\mu_i|^2-\Re\sum_{i=1}^{n}\frac{\rho_i-E+i\eta_0}{(\rho_i-E-i\eta_0)(\rho_i-E+i\eta_0)}|w^T\mu_i|^2 \right| \\
        =~&\left| \sum_{i=1}^{n}\frac{1}{\rho_i-E}|w^T\mu_i|^2-\sum_{i=1}^{n}\frac{\rho_i-E}{(\rho_i-E)^2+\eta_0^2}|w^T\mu_i|^2 \right| \\
        =~& \left| \sum_{i=1}^{n}\frac{\eta_0^2}{(\rho_i-E)((\rho_i-E)^2+\eta_0^2)}|w^T\mu_i|^2 \right| \\
        \leq~& \frac{\eta_0}{E-\lambda_1} \left|\sum_{i=1}^{n}\frac{\eta_0}{(\rho_i-E)^2+\eta_0^2}|w^T\mu_i|^2 \right| \\
        =~& \frac{\eta_0}{E-\lambda_1}|\Im w^TR(z_0)w| \\
        \leq~& \frac{\kappa^{2}\cdot \tilde{C}n^{-1+\epsilon_0}}{\kappa }\cdot \tilde{C}n^{-1+\epsilon_0} \\
        =~& \tilde C^2\kappa n^{-2+2\epsilon_0}~.
    \end{align*}
    Thus, we combine these two inequalities to see
    \begin{align*}
        |w^TR(z)w-w^TR(z_0)w|\leq~& |\Re w^TR(z)w-\Re w^TR(z_0)w|+|\Im w^TR(z)w-\Im w^TR(z_0)w| \\
        =~& |\Re w^TR(z)w-\Re w^TR(z_0)w|+|\Im w^TR(z_0)w| \\
        \leq~& \tilde C n^{-1+\epsilon_0}+\tilde C^2\kappa n^{-2+2\epsilon_0}~,
    \end{align*}
    where we used $|\Im w^TR(z)w|=0$ as $z\in\mathbb R$.
\end{proof}

\subsection{Convergence of Top Eigenvalues}

While the following corollary to Theorem 12 in \cite{sparsehanson} does not appear in their work, the sub-exponential version can be seen as their Theorem 6 and Theorem 7. As pointed out after Theorem 7 in their work, one can prove the corollary by following the argument outlined in \cite[Section 8]{knowleslectures} or applying \cite[Lemma 64]{tao2010randommatricesuniversalitylocal} to \cite[Theorem 12]{sparsehanson}.
\begin{corollary}[of Theorem 12 in \cite{sparsehanson}]\label{eig_interval_cor}
    For any $D>0$, $\delta\in(0,1)$, and interval satisfying $|I|>C_1\frac{\log n}{\delta^4n}$, we have $$\mathbb P\left(\left|\mathcal N_I-n\int_If_{\rm sc}(x)dx\right|\ge\delta n|I|\right)<n^{-D}~,$$ where $\mathcal N_I:=|I \bigcap \rm{spec}(\frac{1}{\sqrt{nq}}W\odot A)|$ and $f_{\rm sc}:=\frac{1}{2\pi}\sqrt{4-x^2}$ is the density of the semicircular law.
\end{corollary}

Thus, we can prove the following lemma, which is necessary in our main proof to control the eigenvalues that remain inside the bulk. 

\begin{lemma}\label{fixed_eig_noise_lemma}
    For any fixed $i\ge1$, we have $\lambda_i(\frac{1}{\sqrt{nq}}W\odot A)\xrightarrow[]{\P}2$. 
\end{lemma}

\begin{proof}
    Assume, for sake of contradiction, that there exists some fixed $i\ge 1$ such that $\lambda_i(\frac{1}{\sqrt{nq}}W\odot A)\not\xrightarrow[]{\P}2$. Then, there exist $\bar \epsilon, \bar  \delta>0$ such that $\mathbb P (\lambda_i(\frac{1}{\sqrt{nq}}W\odot A)\not\in(2-\bar \epsilon, 2+\bar \epsilon):=I)>\bar\delta$. Thus, $\mathcal N_I\le i-1$. 

    Observe the negation of Corollary \ref{eig_interval_cor}: there exist a $D>0, \delta\in(0,1)$, or interval $I$ satisfying $|I|\ge C_1\frac{\log n}{\delta^4 n}$ such that $$\mathbb P\left(\left|\mathcal N_I-n\int_If_{\rm sc}(x)dx\right|\ge\delta n|I|\right)\ge n^{-D}~.$$
    We now show the established interval $I:=(2-\bar \epsilon, 2+\bar \epsilon)$ satisfies the conditions of this statement, thus reaching a contradiction. Let $C_4:=\int_If_{\rm sc}(x)dx$.

    Let $D>0$ be large enough such that $\bar \delta \ge n^{-D}$. Choose $\delta \le \frac{C_4n-(i-1)}{2n\bar \epsilon}\le C_4/(2\bar \epsilon)$ where the latter inequality holds for large enough $n$. Observe $|I|=2\bar \epsilon \ge C_1{ \log n}/({\delta^4n})$ for large enough n. 

    For simplicity's sake, consider values of $n$ large enough such that $C_4n\ge i-1$. Therefore, we have 
    \begin{align*}
        \mathbb P\left(\left|\mathcal N_I-n\int_If_{\rm sc}(x)dx\right|\ge\delta n|I|\right) =~& \mathbb P\left(C_4n-\mathcal N_I\ge2\delta n\bar \epsilon\right) \\
        \ge~& \mathbb P\left(C_4n-\mathcal N_I\ge 2\frac{C_4n-(i-1)}{2n\bar \epsilon}n\bar \epsilon\right) \\
        =~& \mathbb P (C_4n-\mathcal N_I\ge C_4n-(i-1)) \\
        =~& \mathbb P(\mathcal N_I\le i-1) \\
        \ge~& \mathbb P \left(\lambda_i(\frac{1}{\sqrt{nq}}W\odot A)\not \in (2-\bar \epsilon, 2+\bar \epsilon)\right) \\
        >~&\bar\delta \\
        \ge~& n^{-D}~,
    \end{align*}
    thus reaching our desired contradiction. 
\end{proof}

%% file: sections/distinguish.tex
\section{Proofs of Main Results}
Next, we present the proofs of our main results from Section~\ref{sec:main}.
When conditioning on a high probability event, we will use the notation $\mathcal A_j$ where $j$ records the order of creation.

\subsection{Distinguishability Proof}


\begin{proof-of-theorem}[\ref{eigenvalue_theorem}]
    We follow a similar proof strategy to \cite[4.1]{benaychgeorges2010eigenvalues}. Recall the eigenvalues of $X$ are the solutions, $z$, to the equation $\det(X-zI)=0$. We first investigate the eigenvalues $z$ that result from the perturbation. In other words, our analysis considers the case in which the matrix $\frac{1}{\sqrt{nq}}W \odot A-zI$ is invertible. To do so, we consider eigenvalues $z$ which fall outside of the semicircular spectral bulk of our sparse Wigner matrix. To begin with, we define the bulk by picking an $\epsilon<\epsilon_*$ such that $(2+\epsilon_*)<\theta_{r^*}+\frac{1}{\theta_{r^*}}$ where $\theta_{r^*}=\min_{i\in[r]}\{\theta_i:\theta_i>1 \}$. In other words, $\theta_1>\theta_2>\cdots>\theta_{r^*}>1\ge\theta_{r^*+1}>\cdots> \theta_r$.

    
    Let $\mathcal A_1$ be the event $\mathcal A_1:=\{\lambda_1(\frac{1}{\sqrt{nq}}W \odot A) < 2+\epsilon \}$. Going forward, we condition on this event. As $q\gg \frac{\log n}{n}$, Lemma \ref{semicirc_law_sparse_lemma} shows this is a high probability event, and Lemma \ref{high_prob_event} allows us to show all of the following holds with high probability. 

    Suppose that we have an eigenvalue $\lambda_i(X):=z\not\in(-2-\epsilon,2+\epsilon)$. By our choice of $q$, we see $z$ cannot be an eigenvalue of $\frac{1}{\sqrt{nq}}W \odot A$ with high probability and thus $\frac{1}{\sqrt{nq}}W \odot A-zI$ is invertible: 
\begin{align*}
    0=~&\text{det}(X-zI)\\=~&\text{det}(\frac{1}{np}V\Theta V^T+\frac{1}{\sqrt{nq}}W\odot A-zI)\\=~&\text{det}(\frac{1}{\sqrt{nq}}W\odot A-zI)\cdot \text{det}(I+\frac{1}{np}R_n(z)V\Theta V^T)\\=~&\text{det}(\frac{1}{\sqrt{nq}}W\odot A-zI)\cdot \text{det}(I+\frac{1}{np}V^TR_n(z)V\Theta )~,
\end{align*}
where we used Sylvester's Identity ($\text{det}(I+AB)=\text{det}(I+BA)$) in the fourth equality and defined the resolvent to be $R_n(z)=(\frac{1}{\sqrt{nq}}W\odot A-zI)^{-1}$. This result allows us to conclude that 
\begin{align}
    \lambda_i(X)= z\not \in (-2-\epsilon, 2+\epsilon) \iff \text{det}(I+\frac{1}{np}V^TR(z)V\Theta ) = 0~.\label{iff_eig}
\end{align}
We aim to show $\frac{1}{np}V^TR_n(z)V\to m(z)I_r$ where $m(z)$ is the Stieltjes transform of the semicircular distribution, which relies on proving diagonal and off-diagonal concentration. For notational convenience, we henceforth use $R=R_n(z)$, but note the resolvent is a function of the chosen eigenvalue $z$, which scales in our dimension $n$. We begin by considering the diagonal entries. 

\paragraph{Diagonal Concentration.}Adding and subtracting the expectation, we have 
\begin{align}\label{diagonal_resolvent}
    [\frac{1}{np}V^T RV ]_{ii}=\frac{1}{np}v_i^TRv_i=\frac{1}{np}v_i^TRv_i-\mathbb E[\frac{1}{np}v_i^TRv_i]+\mathbb E[\frac{1}{np}v_i^TRv_i]~.
\end{align} 
In order to prove concentration in the first term, we aim to use a variation on the Hanson-Wright Inequality. In order to do so, we must consider the operator and Frobenius norm of our resolvent matrix $R$. By our choice of $z$, we use a bound on the operator norm of the resolvent \cite[Theorem 5.8]{hislop2012introduction} to see 
\begin{align}\label{Resolvent_Op_bound}
    \|R\|_{2} = \left\|(\frac{1}{\sqrt{nq}}W_n\odot A_n - zI)^{-1}\right\|_{2}
    \leq \frac{1}{\text{dist}(\text{spec}(\frac{1}{\sqrt{nq}}W_n\odot A_n), z)}
    \leq \frac{1}{z-2-\epsilon} := \sqrt{C_2}~,
\end{align}
where we notice $C_2$ is some large constant depending on how close our choice of $z$ is to $(2+\epsilon)$. Next, we consider the Frobenius norm. In order to do so, we condition on the structure of  $v_i$. Formally, we define $\mathcal A_2 := \{|\text{supp}(v_i)|=np+k_i:k_i\in[-\log(np)(np)^{1/2},\log(np)(np)^{1/2}]\}$. Lemma \ref{support_size_independent} shows $\mathcal A_2$ is a high probability event, and Lemma \ref{high_prob_event} shows all of the following holds with high probability if we condition on $\mathcal A_2$. We use the notation $R_{\text{supp}(v_i)}$ to denote the $|\text{supp}(v_i)|\times|\text{supp}(v_i)|=(np+k_i)\times (np+k_i)$ submatrix of $R$ corresponding to the support of $v_i$. Note that 
\begin{align*}
    \|R_{\text{supp}(v_i)}\|_F^2 \leq \text{rank}(R_{\text{supp}(v_i)})\|R_{\text{supp}(v_i)}\|_{2}^2 \leq \text{rank}(R_{\text{supp}(v_i)})\|R\|_{2}^2 =~& (np+k_i)\cdot\|R\|_{2}^2
    \le (np+o(np))\cdot C_2 
\end{align*}
where we used the fact that the 2-norm of a submatrix is less than or equal to the operator norm of the full matrix. Thus, for some appropriately chosen $C_3>1$, we have $\|R_{\text{supp}(v_i)}\|_F^2 \leq C_3\cdot C_2\cdot np$ for all $i\in [r]$. Armed with bounds on the operator and Frobenius norms for our submatrix $R_{\text{supp}(v_i)}$, we turn toward the spike vectors. Note $\bar v_i := v_i|_{\text{supp}(v_i)}$ simply reduces the dimension of $v_i$ by removing the $0$ elements. Therefore, it is still a random vector with sub-Gaussian random variable entries. Thus, $\| \bar v_i\|_{\psi_2}\le K$ for some $K>0$. Thus, satisfying the conditions of \cite[Theorem 1.1]{rudelson2013}, we can apply their variation on the Hanson-Wright inequality for our choice of $t={np}/({\log(np))^{\gamma}}$ for some $\gamma>0$. Notice how 
\begin{align*}
    \min \left\{ \frac{t^2}{K^4 \|R_{\text{supp}(v_i)}\|_F^2}, \frac{t}{K^2 \|R_{\text{supp}(v_i)}\|_{2}} \right\} =~& \min \left\{ \frac{(np/\log(np)^{\gamma})^2}{ K^4 \|R_{\text{supp}(v_i)}\|_F^2}, \frac{np/\log(np)^{\gamma}}{ K^2 \|R_{\text{supp}(v_i)}\|_{2}} \right\} \\ 
    \ge~& \min \left\{ \frac{(np/\log(np)^{\gamma})^2}{ K^4 \cdot C_3 C_2 np}, \frac{np/\log(np)^{\gamma}}{ K^2\cdot \sqrt{C_2} } \right\} \\
    =~& \frac{1}{C_3C_2K^4}\frac{np}{(\log(np))^{2\gamma}}~,
\end{align*}
for $n$ large enough. Thus, 
\begin{align*}
    \mathbb P\left( \left| \frac{1}{np} v_i^T R v_i - \frac{1}{np}\mathbb{E}_{v_i} [ v_i^T R v_i]  \right| > \frac{t}{np}\right) =~& \mathbb P( \left| v_i^T R v_i - \mathbb{E}_{v_i} [v_i^T R v_i]  \right| > t) \\
    =~& \mathbb P( \left| v_i^T R_{\text{supp}(v_i)} v_i - \mathbb{E}_{v_i}[v_i^T R_{\text{supp}(v_i)} v_i]  \right| > t) \\
    \leq~& 2\exp\left( -c_1 \min \left\{ \frac{t^2}{K^4 \|R_{\text{supp}(v_i)}\|_F^2}, \frac{t}{K^2 \|R_{\text{supp}(v_i)}\|_{2}} \right\} \right) \\ 
    <~& 2\exp\left( -\frac{c_1}{C_3C_2K^4}\frac{np}{(\log(np))^{2\gamma}} \right)  
    \to~ 0~,
\end{align*}
since, by our assumption, $np\to \infty$. Here, $c_1$ is a positive constant specified in \cite{rudelson2013}. 

We want to show this result holds for all $i\in [r]$. Thus, we can perform the following union bound. We apply this technique multiple times throughout our work. For brevity, we include the computation here and reference the technique going forward. We see
\begin{align}
    \mathbb P \left(\exists i\in[r] :\left|  \frac{1}{np} v_i^T R v_i - \frac{1}{np}\mathbb{E}_{v_i} [ v_i^T R v_i]  \right| > \frac{1}{(\log(np))^{\gamma}} \right) =~&\mathbb P \left(\bigcup_{i=1}^r\left\{\left| \frac{1}{np} v_i^T R v_i - \frac{1}{np}\mathbb{E}_{v_i} [ v_i^T R v_i]  \right| > \frac{1}{(\log(np)^\gamma} \right\}\right) \notag \\
    \le r~&\mathbb P \left(\left| \frac{1}{np} v_i^T R v_i - \frac{1}{np}\mathbb{E}_{v_i} [ v_i^T R v_i]  \right| > \frac{1}{(\log(np))^\gamma}\right) \notag \\
    \le~& 2r\exp\left( -\frac{c_1}{C_3C_2K^4}\frac{np}{(\log(np))^{2\gamma}} \right) 
    \to~ 0   ~.\label{allspikesdiagonal}
\end{align}
We define $$\mathcal A_3 := \left\{\left| \frac{1}{np} v_i^T R v_i - \frac{1}{np}\mathbb{E}_{v_i} [ v_i^T R v_i]  \right| \le (\log np)^{-\gamma} \quad \forall i\in[r]\right\}~,$$
which we will use as a high probability event that we condition on in the proof of Theorem \ref{eigenvector_theorem}.




Moving to the second term of \eqref{diagonal_resolvent}, we are now ready to use Lemma \ref{lemma_2}. Because $q= \tau\frac{\log n}{n}$, we can choose $D>0$ and $\delta=(\tau/C_1)^{-1/4}$ so that $q=\tau(n)\frac{\log n}{n} \geq C_1\frac{\log n}{\delta^4n}$, thus satisfying the conditions of Lemma \ref{lemma_2} and giving us, for sufficiently large $n$, 
\begin{align}
    \mathbb P\left(\max_{1\leq i \leq n}|R_{n}(z)_{ii}-m(z)|\leq \tilde C_1\tau^{-1/4}\right) \geq 1-n^{-D}~,\label{resolvent_control_with_delta_plugged_in}
\end{align}
where we let $\tilde C_1:=\bar C_1C_1^{1/4}$. 

Continuing our conditioning on $\mathcal A_2$, we recall how $\mathbb E [(\bar v_i)_j^2]=1$ and $\mathbb E[(\bar v_i)_j(\bar v_i)_l]=\mathbb E[(v_i)_j]\cdot \mathbb E[(v_i)_l]=0$ for $j,l\in\text{supp}(v_i)$. Thus, we see
\begin{align*}
    \mathbb{E}_{v_i} \left[\frac{1}{np} v_i^T R v_i \right] =~& \mathbb{E}_{v_i} \left[ \text{tr} \left( \frac{1}{np} \bar v_i^T R_{\text{supp}(v_i)} \bar v_i \right) \right] \\
    =~& \frac{1}{np} \text{tr} \left( R_{\text{supp}(v_i)} \mathbb{E}_{v_i} \left[ \bar v_i \bar v_i^T \right] \right) \\
    =~& \frac{1}{np} \text{tr} (R_{\text{supp}(v_i)} I_{\text{supp}(v_i)} ) \\
    =~& \frac{1}{np} \text{tr} (R_{\text{supp}(v_i)}) \\
    =~& \frac{np+k_i}{np}\left( \frac{1}{n p+k_i} \text{tr} (R_{\text{supp}(v_i)})-m(z)+m(z) \right)  \\
    \xrightarrow{\P}~& m(z)~,
\end{align*}
where we used  $|\text{\text{supp}}(v_i)|=n p+k_i $, $\frac{np+k_i}{np}\to 1$ as $k_i\in [ -o(np),o(np) ]$, and  \eqref{resolvent_control_with_delta_plugged_in} to see
\begin{align*}
    1-n^{-D} \leq~& \mathbb P\left(\max_{1\leq i \leq n}|R_{n}(z)_{ii}-m(z)|\leq \tilde C_1\tau^{-1/4}\right) \\
    \leq~& \mathbb{P}\left(\left|\frac{1}{n p+k_i}\text{tr}(R_n(z)_{\text{supp}(v_i)})-m(z)\right|\leq \tilde C_1\tau^{-1/4} \right)~.
\end{align*}
Proving this result for all elements on the diagonal follows a similar process to \eqref{allspikesdiagonal} by taking a union bound: 
\begin{align}
    \mathbb{P}\left(\exists i\in[r] :\left|\frac{1}{n p+k_i}\text{tr}(R_n(z)_{\text{supp}(v_i)})-m(z)\right|> \tilde C_1\tau^{-1/4} \right)\le rn^{-D}\label{resolv_exp_control}~.
\end{align}
As we did with $\mathcal A_3$, we define the following
$$\mathcal A_4 := \left\{\left| \frac{1}{np}\mathbb{E}_{v_i} [ v_i^T R v_i]  \right| \le 2\tilde C_1\tau^{-1/4} +\frac{\log(np)\sqrt{np}}{np} m(z) \quad \forall i\in[r]\right\}~,$$
which, by \eqref{resolv_exp_control}, is a high probability event.




\paragraph{Off-Diagonal Concentration.} Next, we consider the off-diagonal entries of the matrix $\frac{1}{np}V^TRV$. We aim to apply a similar concentration method above and do so by constructing a concatenated matrix $\tilde R$. 
Thus, seeing as the resolvent is symmetric, we have 
\begin{align}
    2[V^TRV]_{ij}=2v_i^TRv_j=~& v_i^TRv_i+2v_i^TRv_j+v_j^TRv_j-v_i^TRv_i-v_j^TRv_j \notag \\
    =~& \begin{bmatrix} v_i \\ v_j\end{bmatrix}^T\begin{bmatrix} R & R \\ R & R\end{bmatrix}\begin{bmatrix} v_i \\ v_j\end{bmatrix} -v_i^TRv_i-v_j^TRv_j \notag \\
    =~& \begin{bmatrix} \bar v_i \\ \bar v_j\end{bmatrix}^T\begin{bmatrix} R_{\text{supp}(v_i)} & R_{\text{supp}(v_i)\cap \text{supp}(v_j)} \\ R_{\text{supp}(v_i)\cap \text{supp}(v_j)} & R_{\text{supp}(v_j)} \end{bmatrix}\begin{bmatrix} \bar v_i \\\bar v_j\end{bmatrix} -v_i^TRv_i-v_j^TRv_j \notag \\
    :=~& \tilde v^T \tilde R \tilde v -v_i^TRv_i-v_j^TRv_j \notag \\
    =~& \tilde v^T \tilde R \tilde v -\mathbb E [\tilde v^T \tilde R \tilde v]+\mathbb E [\tilde v^T \tilde R \tilde v] -v_i^TRv_i-v_j^TRv_j \notag\\
    =~& \tilde v^T \tilde R \tilde v -\mathbb E [\tilde v^T \tilde R \tilde v] + \mathbb E[ v_i^TRv_i]-v_i^TRv_i + \mathbb E[ v_j^TRv_j]-v_j^TRv_j ~,\label{offdiagtilderewrite}
\end{align}
where, in the final equality, we used the independence of the entries of $v_i$ and $v_j$ to see
$$\mathbb E_{v_i,v_j}[v_i^TRv_j]=\mathbb E_{v_i,v_j}[\text{tr}(v_i^TRv_j)]=\text{tr}(R\cdot \mathbb E_{v_i,v_j}[v_jv_i^T])=0~,$$
thus giving us
\begin{align*}
    \mathbb{E}_{v_i, v_j} \left[\tilde v^T \tilde R \tilde v \right] =~& \mathbb{E}_{v_i, v_j} \left[v_i^T R_n(z) v_i \right]+2\mathbb{E}_{v_i, v_j} \left[ v_i^T R_n(z) v_j \right]+\mathbb{E}_{v_i, v_j} \left[ v_j^T R_n(z) v_j \right] \\
    =~& \mathbb{E}_{v_i, v_j} \left[ v_i^T R_n(z) v_i \right]+\mathbb{E}_{v_i, v_j} \left[v_j^T R_n(z) v_j \right]~.
\end{align*}
Observe that, when we properly scale by $\frac{1}{np}$, \eqref{allspikesdiagonal} gives us concentration for the last four terms of \eqref{offdiagtilderewrite}. Thus, we turn our attention toward the first two terms of \eqref{offdiagtilderewrite}. In an effort to apply a similar \cite{rudelson2013} Hanson-Wright Inequality again, we  observe how
\begin{align*}
    \|\tilde R\|_2
    \le~& \| R_{\text{supp}(v_i)} \|_2 + 2\| R_{\text{supp}(v_i)\cap \text{supp}(v_j)} \|_2 +  \|R_{\text{supp}(v_j)} \|_2 \le 4\|R\|_2 \le 4\sqrt{{C_2}}~,
\end{align*}
by our choice of $z\not \in (-2-\epsilon, 2+\epsilon)$. 
Further, we see
\begin{align*}
    \| \tilde R\|_F^2 = \| R_{\text{supp}(v_i)} \|_F^2 + 2\| R_{\text{supp}(v_i)\cap \text{supp}(v_j)} \|_F^2 +  \|R_{\text{supp}(v_j)} \|_F^2 
    \le  4~\max_l (np+k_l)\|R\|_2^2 
    \leq 4C_3 C_2 np~.
\end{align*}
Seeing as the concatenation $\tilde v$ of sub-Gaussian random vectors is also sub-Gaussian, we satisfy the conditions of \cite[Theorem 1.1]{rudelson2013}, and again apply their variation of the Hanson-Wright Inequality with the choice of $t={np}/{(\log(np))^{\gamma}} $ for some $\gamma>0:$
\begin{align}
    \mathbb P\left ( \left| \frac{1}{np} \tilde v^T \tilde R \tilde v - \frac{1}{np}\mathbb{E}_{\tilde v} [ \tilde v^T \tilde R \tilde v]  \right| > \frac{1}{(\log(np))^\gamma}\right)
    \le ~& 2\exp\left( -c_1 \text{min} \left\{ \frac{(np/(\log(np))^{\gamma})^2}{ K^4 \|\tilde R\|_F^2}, \frac{np/(\log(np))^{\gamma}}{ K^2 \|\tilde R\|_{2}} \right\} \right) \notag \\ 
    \le~& 2\exp\left( -c_1 \text{min} \left\{ \frac{(np/(\log(np))^{\gamma})^2}{ K^4 \cdot  4C_3 C_2 np}, \frac{np/(\log(np))^{\gamma}}{  K^2 \cdot 4\sqrt{C_2}} \right\} \right) \notag \\
    =~& 2\exp\left( -\frac{c_1}{4C_3C_2K^4}\frac{np}{(\log(np))^{2\gamma}} \right)\to 0, \label{ineqjsingle}
\end{align}
seeing as $np\to \infty$. 

Therefore, \eqref{offdiagtilderewrite} scaled by ${1}/{np}$ converges to $0$: $\frac{1}{np}v_i^TRv_j\xrightarrow[]{\P}0$. We again perform a standard union bound across all pairings of $i\neq j$ and gain only a multiplicative constant ${r(r-1)}/{2}$ to show $v_i^TRv_j\xrightarrow[]{\P}0$ for all $i \neq j$. As previously done with $\mathcal A_3$ and $\mathcal A_4$, we establish 
\begin{align}
    \mathcal A_5 := \left\{\left| \frac{1}{np} \tilde{v}^T \tilde R \tilde{v} - \frac{1}{np}\mathbb{E}_{v_i,v_j} [ \tilde v^T \tilde R \tilde v]  \right| \le (\log np)^{-\gamma}, \quad \forall \{i,j\}\in[r]^2\right\}~.\label{sixthhighprob}
\end{align}
The above analysis shows this is a high probability event with
\begin{align}
    \mathbb P( \mathcal A_5^c) \le r(r-1)\exp\left(-\frac{c_1}{4C_3C_2K^4}\frac{np}{(\log(np))^{2\gamma}} \right) ~,\label{A_6_ref}
\end{align} which we will exploit in the proof of Theorem \ref{eigenvector_theorem}. Then, using Lemma \ref{high_prob_event}, we perform the standard computation involving our assumption of the high probability events $\mathcal A_1$ and $\mathcal A_2$ to remove conditioning and show the matrix convergence \eqref{matrix_convergence} below occurs with high probability. 

\paragraph{Convergence.} Putting all our results together, we conclude that, entrywise, \begin{align}\label{matrix_convergence}
    I+\frac{1}{np}V^TRV\Theta \xrightarrow[]{\P}I+m(z)\Theta~. 
\end{align} Thus, for a zero to occur on the diagonal, we must have $z=m^{-1}(-1/\theta_i)$. As we only consider $\theta_i>0$ and the range of the Stieltjes transform of the semicircular law  $m(z)=\frac{-z+\sqrt{z^2-4}}{2}$ is $(-1, 0)\cup(1,\infty)$, we see $m(z)=-1/\theta_i$ only when $\theta_i>1$. Solving the equation, we have an eigenvalue $z\not\in(-2-\epsilon,2+\epsilon)$ if and only if $z=\theta _i+\frac{1}{\theta _i}$ for $\theta_i>1$. 

Further, because the same number ($r^*$) of eigenvalues remains outside of the bulk for any choice of $\epsilon<\epsilon^*$, we must have that $\lambda_i(X)\le2$ for $i>r^*$. By Weyl's Interlacing Inequality, we have $$\lambda_{i}(\frac{1}{\sqrt{nq}}W\odot A) \le \lambda_{i}(X)\le 2.$$ By Lemma \ref{fixed_eig_noise_lemma}, we know $\lambda_{i}(\frac{1}{\sqrt{nq}}W\odot A)\xrightarrow[]{\P}2$ and thus $\lambda_{i}(X)\xrightarrow[]{\P} 2$ for fixed $i>r^*$.

\paragraph{Fluctuation.} Finally, we compute the fluctuation of the true eigenvalue $\tilde z$ from the limit eigenvalue $z_i=\theta_i+\frac{1}{\theta_i}$. Recall \eqref{iff_eig}: 
\begin{align*}
    \lambda_i(X)= \tilde z\not \in (-2-\epsilon, 2+\epsilon) \iff \text{det}(I+\frac{1}{np}V^TR(\tilde z)V\Theta ) = 0~.
\end{align*}
We begin by fixing a value of $n>1$ and viewing $I+\frac{1}{np}V^TR(z)V\Theta$ as a function of $z$. We introduce our final high probability event to condition on: $$\mathcal A_6 :=\left\{ \frac{1}{np}\|V\|_2^2\le100r\right\}~,$$ which we prove is a high probability event by Theorem \ref{spike_two_norm_control}. Thus, with high probability, $\|\frac{1}{np}V^TR(z)V\Theta \|_2\le \frac{1}{np}\|V\|_2^2 \|\Theta\|_2\|R(z)\|_2\le 100r\|R(z)\|_2\to 0$ as $z\to \infty$ seeing as $R(z)\to 0$. Therefore, we have $I+\frac{1}{np}V^TR(z)V\Theta\to I$, which ensures that all the eigenvalues of the former converge to $1$ as the eigenvalues are continuous functions of the matrix entries. Consider the largest $\tilde z$ value where $\text{det}(I+\frac{1}{np}V^TR(\tilde z)V\Theta ) = 0$. 
\begin{claim}\label{largest_smallest_claim}
    Given $\tilde z$ is the largest value such that $\text{det}(I+\frac{1}{np}V^TR(\tilde z)V\Theta ) = 0$, we have $\tilde z=\lambda_1(X)$ with high probability and $\lambda_r(I+\frac{1}{np}V^TR(\tilde z)V\Theta)=0$.
\end{claim}
\begin{proof}
    $\tilde z=\lambda_1(X)$ is an immediate consequence of \eqref{iff_eig}, which we restated above. For the second part of the claim, consider, for sake of contradiction, that there was another eigenvalue of $I+\frac{1}{np}V^TR(\tilde z)V\Theta$ that was smaller (in eigenvalue ordering). Then, this eigenvalue would have to be negative, and, as all the eigenvalues converge to $1$, we could find a value of $z$ greater than $\tilde z$ where this smaller, negative eigenvalue becomes zero by continuity. We would have a $z>\tilde z$ for which the determinant equals zero, which is a contradiction. Thus, $\lambda_r(I+\frac{1}{np}V^TRV\Theta)=0$ and $\lambda_i(I+\frac{1}{np}V^TRV\Theta)>0$ for $i<r$. 
\end{proof} 

Earlier in this proof, we find \eqref{matrix_convergence}, which in turn gives us convergence of $\tilde z$ to the value $z_1=\theta_1+\frac{1}{\theta_1}$. As we are working with real symmetric matrices, we can use Weyl's Theorem to see
\begin{align*}
    \left| \lambda_i\left (I + m(\tilde z)\Theta\right)-\lambda_i \left(I+\frac{1}{np}V^TR(\tilde z)V\Theta \right) \right| \le \| \mathcal L \|_2\le  \| \mathcal L \|_F\le \Lambda ~,
\end{align*}
where $\mathcal L= m(\tilde z)\Theta-\frac{1}{np}V^TRV\Theta$, i.e. the vanishing terms that are left over after removing the expected terms from the matrix of interest. 
We refer the reader to Appendix \ref{vanishing_terms_frob_bound} to see a detailed bound on $\| \mathcal L \|_F^2$. Going forward, we use $$\Lambda = \|\Theta\|_F\sqrt{C_5r^2\max\{(\log(np))^{-2\gamma}, ~ \tau^{-1/2}\}}$$ to represent this computed bound. 

For the former matrix, we know its eigenvalues are on the diagonal. Because, for any $\tilde z>(2+\epsilon),$ we know the diagonal entries obey the inequality $1+m(\tilde z)\theta_i<1+m(\tilde z)\theta_{i+1}$, we see $\lambda_r(I+m(\tilde z)\Theta)=1+\theta_1m(\tilde z)$. Thus, we have $$|\lambda_r(I+m(\tilde z)\Theta)-\lambda_r(I+\frac{1}{np}V^TRV\Theta)|=|1+\theta_1m(\tilde z)| \le \Lambda~.$$
Therefore, because $m(z)$ is a strictly increasing function, we have $$m^{-1}((-1-\Lambda)/\theta_1) \le \tilde z \le m^{-1}((\Lambda -1)/\theta_1)~,$$
where we consider $n$ large enough so that $\Lambda\le \min\{\theta_1-1,1 \}$, ensuring the above terms fall within the domain of the inverse function. Because $m^{-1}(z)=-z-\frac{1}{z},$ we can compute these values specifically: 
$$\frac{1+\Lambda}{\theta_1}+\frac{\theta_1}{1+\Lambda} \le \tilde z \le\frac{1-\Lambda}{\theta_1}+\frac{\theta_1}{1-\Lambda}\implies \frac{\Lambda}{\theta_1}-\frac{\theta_1\Lambda}{1+\Lambda} \le \tilde z-z_1 \le\frac{-\Lambda}{\theta_1}+\frac{\theta_1\Lambda}{1-\Lambda}~.$$
Because $\Lambda \le 1,$ we can choose $C_7>1$ such that $\Lambda \le (C_7-1)/C_7$ implies
\begin{align}\label{fluctuation_bound}
    -C_7\theta_1\Lambda \le \tilde z-z_1 \le C_7\theta_1\Lambda \implies |\tilde z - z_1| \le C_7\theta_1\Lambda~.
\end{align}

Next, we can repeat this process for the next largest $\tilde z_2$ value where the determinant equals zero. A similar argument made in Claim \ref{largest_smallest_claim} can be made for this new $\tilde z_2$ value to see that $\tilde z_2=\lambda_2(X)$ could only correspond with $\lambda_r(I+\frac{1}{np}V^TRV\Theta)$ or $\lambda_{r-1}(I+\frac{1}{np}V^TRV\Theta)$. Because we have already found $\tilde z_1$ to correspond to $\lambda_r(I+\frac{1}{np}V^TRV\Theta)$, the one-to-one nature between the eigenvalues allows us to conclude $\lambda_2(X)=\lambda_{r-1}(I+\frac{1}{np}V^TRV\Theta)$, therefore allowing us to perform the above analysis once again and extend this result to the second signal and all signals thereafter. Hence, we have 
\begin{align*}
    &\mathbb P \left( \left | \lambda_i(X)-(\theta_i+\frac{1}{\theta_i}) \right| \le C_7\theta_i\Lambda \right) \\
    \ge & ~1-\max\left\{r(r-1)\exp\left(-\frac{c_1}{4C_3C_2K^4}\frac{np}{(\log(np))^{2\gamma}}\right),~  rn^{\max\{-2,-D\}},~ \exp(-C_8nq)\right\}~,
\end{align*}
for $\theta_i>1$ where the probability terms come from removing the conditioning on $\mathcal A_i$, and we redefine $C_2:=\max\{C_2,1\}$ to take into account the very similar tail bound between $\mathcal A_3$ and $\mathcal A_6$.

\end{proof-of-theorem}




%% file: sections/recover.tex
\subsection{Recovery Proof}

Next, we show the recovery result. We use as inspiration the blueprint provided by \cite{benaychgeorges2010eigenvalues}; however, since we do not have orthogonal invariance, which is crucial in their procedure, we must adopt a new strategy. 



\begin{proof-of-theorem}[\ref{eigenvector_theorem}]
    We continue with our choice of $\epsilon < \epsilon_*$ from the beginning of Theorem \ref{eigenvalue_theorem}. Without loss of generality, we perform our computations for the eigenvalue $\tilde z_1\to z_1=\theta _1+\frac{1}{\theta _1}$ and note the same result can be concluded for any choice of $z_i$ as long as $\theta _i>1$. Because $\theta_1>1$, we have $\tilde z _1\not\in(-2-\epsilon, 2+\epsilon)$ for $n$ large enough. 
    
    We continue to condition on the high probability event $\mathcal A_1 = \{\lambda_1(\frac{1}{\sqrt{nq}}W\odot A)<2+\epsilon \}$. From Theorem \ref{eigenvalue_theorem}, we have that $\tilde z_1$ is an eigenvalue of $X$ and assume $Xu_1=\tilde z_1u_1$ where $(\tilde z_1,u_1)$ is the normalized eigenpair corresponding to the recovered eigenvalue associated with $\theta _1$ in Theorem \ref{eigenvalue_theorem}. 
Observe how 
\begin{align}
    Xu_1=\tilde z_1u_1 \iff~& (\frac{1}{np}V\Theta V^T+\frac{1}{\sqrt{nq}}W\odot A)u_1=\tilde z_1 u_1 \notag \\ \iff~& u_1 = -\frac{1}{np}RV\Theta V^Tu_1 \label{u_1_rewrite}\\
    \implies~& \frac{1}{\sqrt{np}}V^Tu_1=-\frac{1}{\sqrt{np}}\frac{1}{np}V^TRV\Theta V^Tu_1 \notag \\
    \iff~& (I+\frac{1}{np}V^TRV\Theta )\frac{1}{\sqrt{np}}V^Tu_1=0~. \notag
\end{align}
Each of these matrices and vectors are $n$-dependent. The fact that $\frac{1}{\sqrt{np}}V^Tu_1\in \mathcal N(I+\frac{1}{np}V^TRV\Theta )$ implies that the limit point of $\frac{1}{\sqrt{np}}V^Tu_1$ is in the Kernel of $\text{diag}(0,1-\frac{\theta _2}{\theta _1}, ...1-\frac{\theta _r}{\theta _1})$ by continuity of matrix multiplication and \eqref{matrix_convergence}:
\begin{align*}
    I+\frac{1}{np}V^TRV\Theta \to\text{diag}\left(0,~1-\frac{\theta _2}{\theta _1}, ...,~1-\frac{\theta _r}{\theta _1}\right)~.
\end{align*}
Therefore, 
\begin{align}
    \frac{1}{\sqrt{np}}V^Tu_1\to \mathcal Z_1e_1~,\label{inner_prod_conv}
\end{align}i.e. $\langle u_1,  \frac{1}{\sqrt{np}} v_j \rangle \to 0$ for $j>1$ and $\langle u_1, \frac{1}{\sqrt{np}}v_1 \rangle \to \mathcal Z_1>0$. We aim to find the exact value for $\mathcal Z_1,$ but, to do so, we need to establish two key results. The first result involves establishing a specific rate of decay for the $j>1$ terms $\langle u_1,  \frac{1}{\sqrt{np}} v_j \rangle \to 0$. We perform perturbation analysis similar to the end of the proof of Theorem \ref{eigenvalue_theorem}. Observe the following variation on the Davis--Kahan Theorem: 
\begin{proposition}[Corollary 3 of \cite{yu2015useful}]
    Let $A$ and $B$ be symmetric, real matrices in $\mathbb{R}^{n\times n}$. For each $i\in [n]$, suppose that the smallest eigengap $\tilde \delta = \min\{|\lambda_{i-1}(A)-\lambda_i(A)|,|\lambda_i(A)-\lambda_{i+1}(A)|\}>0$. Then, the corresponding eigenvectors of $A$ and $B$ satisfy $$\min_{s\in{\{\pm1\}}}\|v_i(A)-sv_i(B)\|^2\le \frac{8\|A-B\|^2}{\tilde \delta^2},$$ where $v_i$ are the orthonormal eigenvectors of their respective matrices.
\end{proposition}
Using this proposition, we see
\begin{align*}
    \min_{s\in\{\pm1\}}\left\| v_i\left (I + m(\tilde z)\Theta\right)-sv_i \left(I+\frac{1}{np}V^TR(\tilde z)V\Theta \right) \right\|^2 \le \frac{8\|\mathcal L\|^2}{\tilde \delta^2}\le\frac{8\Lambda^2}{\tilde \delta^2}~,
\end{align*}
where $\mathcal L= m(\tilde z)\Theta-\frac{1}{np}V^TRV\Theta$ and $\Lambda = \|\Theta\|_F\sqrt{C_5r^2\max\{(\log(np))^{-2\gamma}, ~ \tau^{-1/2}\}}$ from Appendix \ref{vanishing_terms_frob_bound}.
We are specifically interested in $i=r$ and refer the reader to Claim \ref{largest_smallest_claim} at the end of Theorem \ref{eigenvalue_theorem} to see an argument on the eigenvalue/eigenvector ordering. With this established ordering, we have
\begin{align*}
    \min_{s\in\{\pm1\}}\left\| e_1-s\cdot\|\frac{1}{\sqrt{np}}V^Tu_1\|_2^{-1}\frac{1}{\sqrt{np}}V^Tu_1 \right\|^2 \le\frac{8\Lambda^2}{\tilde \delta^2}~,
\end{align*}
and therefore
\begin{align*}
    \min_{s\in\{\pm1\}}(\langle u_1, \frac{1}{\sqrt{np}}v_1\rangle-\|\frac{1}{\sqrt{np}}V^Tu_1\|_2s)^2+\sum_{i=2}^{r}\langle u_1, \frac{1}{\sqrt{np}}v_i\rangle^2 \le\frac{8\Lambda^2}{\tilde \delta^2}\|\frac{1}{\sqrt{np}}V^Tu_1\|_2^{2}\le\frac{8\Lambda^2}{\tilde \delta^2}r~.
\end{align*}
Notice all the terms on the left hand side are positive, so we can lower bound and establish specific control over each of these individual terms. Now, while we cannot establish specific error bounds on the $j=1$ inner product (in fact, this is the very result we are working toward) due to the scaling factor from the normalization, we do have $\langle u_1, \frac{1}{\sqrt{np}}v_1 \rangle \to \mathcal Z_1>0$, so we know it is of a constant order that does not blow up to infinity. Therefore, we can conclude
\begin{align}
    |\langle u_1, \frac{1}{\sqrt{np}}v_1 \rangle|= O(1) \text{ and } |\langle u_1, \frac{1}{\sqrt{np}}v_j\rangle|= O(\Lambda) \text{ for } j>1~. \label{eigv_fluc_1}
\end{align}

Now, moving toward finding the explicit value of $\mathcal Z_1$, we establish our second key result, which involves control over the derivative of our resolvent terms.
We introduce the following proposition, which is a restatement of Vitali's convergence theorem that allows us to extend the convergence \eqref{matrix_convergence} from Theorem \ref{eigenvalue_theorem} to the derivative of the Stieltjes transform. We prove the following for any $z\not \in(-2-\epsilon,2+\epsilon)$.

\begin{proposition}\label{stiel_deriv_conv_lemma}
    (\cite[Lemma 2.14]{baisilverstein}). Let $f_1 , f_2 , . . .$ be analytic on the domain $\mathcal D$, satisfying $|f_n(z)| \leq M$ for every $n$ and $z$ in $\mathcal D$, and $f_n(z)$ converges for each $z$ in a subset of $\mathcal D$ having a limit point in $\mathcal D$. Then there exists an analytic function $f$ on $\mathcal D$ such that $f_n(z) \to f(z)$ and  $f_n'(z) \to f'(z)$ for all $z\in \mathcal D$. Moreover, for any set bounded by a contour $\Gamma$ interior to D, the convergence is uniform and $f_n'(z)$ are uniformly bounded. 
\end{proposition}

Define $f_{n,i,j}=\frac{1}{np}v_i^TRv_j$. In Theorem \ref{eigenvalue_theorem}, we showed that $f_{n,i,i}(z)\to m(z)$ and $f_{n,i,j}\to 0$ for $i\neq j$ for any value of $z\not \in (-2-\epsilon, 2+\epsilon)$. Thus, we intend to use Proposition \ref{stiel_deriv_conv_lemma}, which concludes $f_{n,i,i}'(z)\to m'(z)$ and $f_{n,i,j}'\to 0$ for $i\neq j$. To satisfy the conditions of the lemma, we analyze the behavior of the functions. 
\paragraph{$f_{n,i,j}$ Are Analytic.} First, we define the domain to be $\mathcal D = \mathbb R\backslash(-2-\epsilon, 2+\epsilon)$. Thus, our eigenvalue of consideration $z_1\in \mathcal D$. Further, fixing $\{i,j \} \in [r]^2$, $f_{n,i,j}$ are analytic functions for all $n$ seeing as we condition on the high probability event $\mathcal A_1$ where $z$ is not an eigenvalue of $\frac{1}{\sqrt{nq}}W \odot A$ by our choice of domain, meaning no $z \in \mathcal D$ can be a pole of the resolvent $R_n(z)$. 
\paragraph{$f_{n,i,j}$ Are Uniformly Bounded.} Next, we need to show $f_{n,i,j}$ is uniformly bounded. We recall and condition on the following three events, which we defined and showed to be high probability events in Theorem \ref{eigenvalue_theorem}:  
$$\mathcal A_3 = \left\{\left| \frac{1}{np} v_i^T R v_i - \frac{1}{np}\mathbb{E}_{v_i} [ v_i^T R v_i]  \right| \le (\log np)^{-\gamma} \quad \forall i\in[r]\right\}\footnote{Notice this choice of $t$ from the Hanson Wright Inequality does cause issues at $n=1$, seeing as the function blows up. However, as we are only concerned with results in an asymptotic sense, we begin our indexing at $n=2$ for our sequence of functions.}$$ 

$$\mathcal A_4 = \left\{\left| \frac{1}{np}\mathbb{E}_{v_i} [ v_i^T R v_i]  \right| \le \tilde 2C_1\tau^{-1/4} + \frac{np+\log(np)\sqrt{np}}{np} m(z) \quad \forall i\in[r]\right\}$$ 

$$\mathcal A_5 = \left\{\left| \frac{1}{np} \tilde{q}^T \tilde R \tilde{q} - \frac{1}{np}\mathbb{E}_{v_i,v_j} [ \tilde q^T \tilde R \tilde q]  \right| \le (\log np)^{-\gamma} \quad \forall \{i,j\}\in[r]^2\right\}.$$ First, consider the diagonal terms $i=j\in [r]$. 
With our conditioning on $\mathcal A_3$ and $\mathcal A_4$, we see 
    \begin{align}
        |f_{n,i,i}| \le (\log np)^{-\gamma} + 2\tilde C_1\tau^{-1/4} + \frac{np+\log(np)\sqrt{np}}{np}m(z)  =m(z)+O(\max\{(\log(np))^{-\gamma}, ~ \tau^{-1/4}\})~.\label{diagonal_terms_explicit}
    \end{align}
Thus, $|f_{n,i,i}|\le M$ for some appropriately chosen constant $M$ making use of the fact that the Stieltjes transform is a uniformly bounded function on $\mathcal D$. 

Next, consider the off-diagonal terms $i \neq j$. Using \eqref{offdiagtilderewrite} and our conditioning on $\mathcal A_4$ and $\mathcal A_5$, we see
\begin{align}
    |f_{n,i,j}|=|\frac{1}{np}v_i^TRv_j|=~&\frac{1}{2}\left|\frac{1}{np}\tilde q^T \tilde R \tilde q -\frac{1}{np}\mathbb E [\tilde q^T \tilde R \tilde q]+\frac{1}{np}\mathbb E [\tilde q^T \tilde R \tilde q] -\frac{1}{np}v_i^TRv_i-\frac{1}{np}v_j^TRv_j\right| \notag \\ 
    =~&\frac{1}{2}\left|\frac{1}{np}\tilde q^T \tilde R \tilde q -\frac{1}{np}\mathbb E [\tilde q^T \tilde R \tilde q]+ \mathbb E[v_i^TRv_i]+\frac{1}{np}\mathbb E[v_j^TRv_j] -\frac{1}{np}v_i^TRv_i-\frac{1}{np}v_j^TRv_j\right| \notag \\
    \le~&\frac{1}{np}|\tilde q^T \tilde R \tilde q -\mathbb E [\tilde q^T \tilde R \tilde q]|+\frac{1}{np}|v_i^TRv_i-\mathbb E[v_i^T Rv_i]| + \frac{1}{np}|v_j^TRv_j-\mathbb E[v_j^T Rv_j]| \notag \\
    \le~& (\log np)^{-\gamma}+2(\log np)^{-\gamma} \notag \\
    =~&3(\log np)^{-\gamma}\notag \\
    =~& O((\log(np))^{-\gamma})~.\label{offdiag_explicit}
\end{align}

Thus,  $|f_{n,i,j}|\le 3M$ for all $z \in \mathcal D$ and $n>1$ and is thus uniformly bounded. Now armed with the conditions to properly apply the proposition, we have $f_{n,i,i}(\tilde z_1)\to m'(\tilde z_1)$ and $f_{n,i,i}( z_1)\to m'( z_1)$ where $\tilde z_1 \to z_1$. We can conclude $f'_{n,i,i}(\tilde z_1)\to m'( z_1)$ if the convergence from the proposition is uniform by the Uniform Limit Theorem, which, by the proposition, we have if we carefully choose a bounded interval $\Gamma$ of $\mathbb R\backslash (-2-\epsilon,2+\epsilon)$ containing all the values of $\tilde z_1$ and $z_1$. In the spirit of this result, we can use this same bounded contour $\Gamma$ containing all of our values of $\tilde z_i$ and $z_i$ to apply the Cauchy Integral Formula to achieve concentration bounds for the derivative of our functions. Therefore, from \eqref{diagonal_terms_explicit} and \eqref{offdiag_explicit}, we have
\begin{align}
    |f'_{n,i,i}(z)-m'(z)|= O(\max\{(\log(np))^{-\gamma}, ~ \tau^{-1/4}\}) \text{ and } |f'_{n,i,j}(z)|= O((\log np)^{-\gamma}) ~,\label{eigv_fluc_2}
\end{align}
thus establishing our second and final ingredient needed to find $\mathcal Z_1$.

From \eqref{u_1_rewrite}, we know 
\begin{align*}
    -u_1=~&\frac{1}{np}RV\Theta V^Tu_1=\frac{1}{np}R\begin{bmatrix}
        \theta_1v_1 & \theta_2v_2 &  \cdots &  \theta_r v_r
    \end{bmatrix}\begin{bmatrix}
        \langle u_1, v_1 \rangle \\ \langle u_1, v_2 \rangle \\ \vdots \\ \langle u_1, v_r \rangle
    \end{bmatrix}
    =R\frac{\theta _1}{np}\langle u_1, v_1 \rangle v_1+R\sum_{j=2}^{r}\frac{\theta _j}{np}\langle u_1, v_j \rangle v_j~
\end{align*}
Thus, using \eqref{eigv_fluc_1} and  \eqref{eigv_fluc_2} and taking the norm squared of $u_1$, we see
{\small
\begin{align*}
    \|u_1\|^2 =~&\sum_{j=1}^{r}\frac{\theta _j^2}{(np)^2}|\langle u_1, v_j \rangle|^2v_j^TR^2 v_j +\frac{\theta _1}{np}\langle u_1, v_1 \rangle \sum_{j=2}^{r}\frac{\theta _j}{np}\langle u_1, v_j \rangle v_1^TR^2v_j +\sum_{j, l=2 j\neq l}^{r}\frac{\theta _j}{(np)^2}\theta _l \langle u_1, v_j \rangle \langle u_1, v_l \rangle v_j^TR^2v_l \\
    =~& \sum_{j=1}^{r}\theta _j^2|\langle u_1, \frac{1}{\sqrt{np}}v_j \rangle|^2 (-f_{n,j,j}')+\theta _1\langle u_1, \frac{1}{\sqrt{np}}v_1 \rangle \sum_{j=2}^{r}\theta _j \langle u_1, \frac{1}{\sqrt{np}} v_j \rangle (-f'_{n,1,j}) + \cdots \\ 
    ~& +\sum_{\substack{j, l=2, \\ j\not = l}}^{r} \theta _j \theta _l \langle u_1, \frac{1}{\sqrt{np}}v_j \rangle \langle u_1, \frac{1}{\sqrt{np}} v_l \rangle  (-f'_{n,j,l}) \\
    =~& \sum_{j=1}^{r}\theta _j^2|\langle u_1, \frac{1}{\sqrt{np}}v_j \rangle|^2\cdot (m'(z_1)+\mathcal M_j)+ O(1)\sum_{j=2}^{r} O(\Lambda) \cdot O((\log np)^{-\gamma})
    +\sum_{\substack{j, l=2, \\j\not = l}}^{r}  O(\Lambda) \cdot O(\Lambda)\cdot   O((\log np)^{-\gamma})  \\
    =~& \| \frac{1}{\sqrt{np}}\tilde\Theta V^Tu_1 \|_2^2 +O(\Lambda(\log np)^{-\gamma})~.
\end{align*}
}
where we let $$\tilde \Theta = \text{diag}(\theta_1\sqrt{m'(z_1)+\mathcal M_1}, ~\theta_2\sqrt{m'(z_1)+\mathcal M_2}, ~ \cdots, ~\theta_r\sqrt{m'(z_1)+\mathcal M_r})~,$$ with $\mathcal M_j=O(\max\{(\log(np))^{-\gamma}, ~ \tau^{-1/4}\})$ by \eqref{eigv_fluc_2}. As $\|u_1\|^2= 1$, we see $$\| \frac{1}{\sqrt{np}}\tilde \Theta V^Tu_1 \|_2^2=1+O(\Lambda(\log np)^{-\gamma })~.$$ 
Therefore, for our final fluctuation result, we find
{\small
\begin{align*}
    \left|\theta_1^2\langle u_1, \frac{1}{\sqrt{np}}v_1 \rangle ^2-\frac{1}{m'(z_1)} \right|=~&\frac{1}{m'(z_1)}\left|\theta_1^2m'(z_1)\langle u_1, \frac{1}{\sqrt{np}}v_1 \rangle ^2+\theta_1^2 \mathcal M_1\langle u_1, \frac{1}{\sqrt{np}}v_1 \rangle ^2-\theta_1^2 \mathcal M_1\langle u_1, \frac{1}{\sqrt{np}}v_1 \rangle ^2-1 \right| \\
    \le~&\frac{1}{m'(z_1)}\left|\| \frac{1}{\sqrt{np}}\tilde\Theta V^T u_1\|^2-\sum_{j=2}^r\left(\theta_j^2(m'(z_1)+\mathcal M_j)\langle u_1,\frac{1}{\sqrt{np}}v_j\rangle^2\right)-1 \right| + O(1)|\mathcal M_1 |\\
    \le~&\frac{1}{m'(z_1)}\left|\| \frac{1}{\sqrt{np}}\tilde\Theta V^T u_1\|^2-1 \right| +\frac{r-1}{m'(z_1)}\left(m'(z_1)+\mathcal M_{\max}\right)\cdot O(\Lambda^2) + O(1)|\mathcal M_1 |\\
    \le~& O(\Lambda (\log np)^{-\gamma})+O(\Lambda^2)+O(\max\{(\log(np))^{-\gamma}, \tau^{-1/4}\}\Lambda^2)+O(\max\{(\log(np))^{-\gamma}, \tau^{-1/4}\}) \\
    =~&O(\max\{(\log(np))^{-\gamma}, ~ \tau^{-1/4}\})~.
\end{align*}}
where we recall $\Lambda = \|\Theta\|_F\sqrt{C_5r^2\max\{(\log(np))^{-2\gamma}, ~ \tau^{-1/2}\}}$.
Therefore, with our established fluctuation result, we can solve for the derivative of the Stieltjes transform at $z_1=\theta_1+\frac{1}{\theta_1}$ to see $m'(\theta_1+\frac{1}{\theta_1})=\frac{1}{\theta_1^2-1}$. Thus,
\begin{align*}
    \begin{bmatrix}
        \langle u_1, \frac{1}{\sqrt{np}}v_1 \rangle \\ \langle u_1, \frac{1}{\sqrt{np} }v_2 \rangle \\ \vdots \\ \langle u_1, \frac{1}{\sqrt{np}}v_r \rangle
    \end{bmatrix}
    = \frac{1}{\sqrt{np}}V^Tu_1 \to \begin{bmatrix}
        \sqrt{1-\frac{1}{\theta _1^2}} \\ 0 \\ \vdots \\ 0
    \end{bmatrix} \text{ for } \theta _1>1~.
\end{align*}
Notice we then have $|\langle u_1, \frac{1}{\sqrt{np}}v_1 \rangle|^2\to1-\frac{1}{\theta_1^2}>0$. Here, $\mathcal A_1$ occurs with high probability by Lemma~\ref{semicirc_law_sparse_lemma}. Events $\mathcal A_3, \mathcal A_4,$ and $\mathcal A_5$ occur with high probability by equations \eqref{allspikesdiagonal}, \eqref{resolv_exp_control}, and \eqref{A_6_ref}, respectively. We can thus remove our conditioning and use Lemma \ref{high_prob_event} to conclude our final result holds with high probability.

We remark that a similar result holds when we allow signals to equal each other; see Appendix \ref{equal_signals}.

\end{proof-of-theorem}




%% file: sections/appendix.tex
\subsection{High Probability Conditioning}
We begin by establishing a collection of lemmas which we will use frequently throughout the work. 

First, we include the following lemma used to remove conditioning on high probability events $\mathcal A_i$. 

\begin{lemma}\label{high_prob_event}
    Let $\mathcal A$ be a high probability event, i.e. $\mathbb P (\mathcal A)\to 1$. For event $\mathcal E$, if $\mathbb P (\mathcal E|\mathcal A)\to 1$, then $\mathbb P(\mathcal E)\to 1$.
\end{lemma}

\begin{proof}
    Notice $\mathbb P(\mathcal E)=\mathbb P (\mathcal E| \mathcal A)\cdot \mathbb P(\mathcal A)+\mathbb P (\mathcal E| \mathcal A^c)\cdot \mathbb P(\mathcal A^c)\ge \mathbb P (\mathcal E| \mathcal A)\cdot \mathbb P(\mathcal A)\to1$.
\end{proof}

Now, we are interested in the following high probability result. When proving probabilistic results involving the spike, we often condition on the size of the support being known. 

\subsection{Support Size Concentration of Spike Vectors}

The model in \eqref{model} allows for entries of the spike to contain Bernoulli random variables with parameter $p$. Thus, the expected size of the support of $v_i$ would be $np$; however, fluctuation can occur. In order to condition on the support being fixed, we show that the fluctuation is $o(np)$ with high probability. 

\begin{lemma}\label{support_size_independent}
    For $r$ independent vectors $v_i=b_i\odot\tilde v_i$ and $np\to\infty$, we have $$\mathbb P \left(\exists i\in[r]: |\text{supp}(v_i)|\not\in \left[np-\log(np)\sqrt{np}, np+\log(np)\sqrt{np}\right] \right)\le {r}(np)^{-2}\to 0~.$$
\end{lemma}

\begin{proof}
    Let $S_n:=\sum_{j=1}^{n}|(b_i)_j|$. We apply the Central Limit Theorem to prove $P(S_n \in [np - k, np + k]) \to 1,$ as $n \to \infty$ for $k=\log(np)\sqrt{np}$:
    \begin{align}
        P\left(\frac{-k}{\sqrt{np(1-p)}}\le \frac{S_n-np}{\sqrt{np(1-p)}}\le\frac{k}{\sqrt{np(1-p)}}\right)\to 1~,
    \end{align}
    seeing as $\frac{k}{\sqrt{np(1-p)}}=\frac{\log(np)}{\sqrt{1-p}}\to \infty$ since $\frac{S_n-np}{\sqrt{np(1-p)}}\to\mathcal{N}(0,1)$ in distribution.

    Thus, with high probability, $|\text{supp}(v_i)|\in\left[np-\log(np)\sqrt{np},np+\log(np)\sqrt{np}\right]$. 

    Note we can use a standard Gaussian tail bound to see $\mathbb P (|\text{supp}(v_i)|\not \in [np-k,np+k])\to0~,$ seeing as $$\mathbb P (|\text{supp}(v_i)|\not \in [np-k,np+k])\to \mathbb P\left(|\mathcal N(0,1)|\ge \frac{k}{\sqrt{np(1-p)}}\right)~,$$ and $$\mathbb P\left(|\mathcal N(0,1)|\ge \frac{k}{\sqrt{np(1-p)}}\right)\le 2\exp \left(-(\log(np)/(1-p))^2/2\right)\le(np)^{-2}\to 0~.$$
    Choosing $k_i=\log(np)\sqrt{np}$ for all $i\in[r]$, we use a union bound to see $P (\exists i\in[r] : |\text{supp}(v_i)|\not \in [np-k_i,np+k_i])\le rP (|\text{supp}(v_i)|\not \in [np-k,np+k])\to 0$ seeing as $$r\mathbb P(|\mathcal N(0,1)|\ge \frac{k}{\sqrt{np(1-p)}})\le 2r\exp \left (-(\log(np)/(1-p))^2/2\right)\to0.$$

\end{proof}

\subsection{Norm Concentration of Sparse Spike}
\begin{theorem}\label{spike_two_norm_control}
    We have $$\mathbb P\left(\frac{1}{np}\|V\|_2^2-r\le r\frac{1}{np(\log(np)^\gamma)}+r\frac{\log(np)\sqrt{np}}{np}\right)\ge1-2r\exp\left( -\frac{c_1}{C_3K^4}\frac{np}{(\log(np))^{2\gamma}} \right)~.$$
\end{theorem}
\begin{proof}
    We apply the same procedure as the diagonal concentration section in Theorem \ref{eigenvalue_theorem} proof but using $R=I$. Then, conditioning on $\mathcal A_2$, $C_2=1$ and $||I_{\text{supp}(v_i)}||_F^2=np+k_i$. Thus,
\begin{align}
    \mathbb P \left(\exists i\in[r] :\left|  \frac{1}{np} v_i^T v_i - \frac{1}{np}\mathbb{E}_{v_i} [ v_i^T v_i]  \right| > \frac{1}{(\log(np))^{\gamma}} \right)\le~& 2r\exp\left( -\frac{c_1}{C_3K^4}\frac{np}{(\log(np))^{2\gamma}} \right) 
    \to~ 0   ~.
\end{align}
and
$\frac{1}{np}\mathbb{E}_{v_i} [ v_i^T v_i]=\frac{1}{np}I_{\text{supp}(v_i)}$. Thus, 
\begin{align*}
    \frac{1}{np}\|V\|_2^2\le\frac{1}{np}\|V\|_F^2=\frac{1}{np}\text{Tr}(V^TV)\le\frac{1}{np}\sum_{i=1}^{r}|v_i^Tv_i| 
    \le \frac{1}{np}r\left((\log(np)^{-\gamma}+ np+k_{\max}\right) ~.
\end{align*}
\end{proof}
\subsection{Bounding The Vanishing Terms}\label{vanishing_terms_frob_bound}
\begin{lemma}
    Let $\mathcal L = m(\tilde z)\Theta-\frac{1}{np}V^TR(\tilde z)V\Theta$. Then, we have $$\|\mathcal L \|_F^2 \le \| \Theta\|_F^2\left(C_5r^2\max\{(\log(np))^{-2\gamma}, ~ \tau^{-1/2}\} \right)~.$$
\end{lemma}

\begin{proof}
    Let 
    \begin{align*}
        L_i=\left |\frac{1}{np}v_i^TRv_i-\mathbb E [\frac{1}{np}v_i^TRv_i]\right|, ~
        M_i =~& \left| \frac{np+k_i}{np}\mathbb E\left[ \frac{1}{np+k_i}v_i^TRv_i\right]-\frac{np+k_i}{np}m(\tilde z)\right|, 
        N_i = \left|\frac{k_i}{np}m(\tilde z)\right|~, \\  \text{ and }
        \tilde L_{ij} =~& \left |\frac{1}{np}\tilde v^T\tilde R \tilde v-\mathbb E [\frac{1}{np}\tilde v^T\tilde R\tilde v]\right|~.
    \end{align*} 
    We begin with the off-diagonal terms. From \eqref{offdiagtilderewrite} and $\mathcal A_5$, we see
    \begin{align*}
        \left |2v_iRv_j \right |^2 \le \tilde L_{ij}^2 +L_i^2+L_j^2 + \tilde L_{ij}L_i+\tilde L_{ij}L_j+L_iL_j 
        \le 6(\log np)^{-2\gamma}~.
    \end{align*}
    Next, we consider the diagonal terms. From $\mathcal A_3$ and $ \mathcal A_4,$ we see
    \begin{align*}
        \left |\frac{1}{np}v_i^TRv_i-m(\tilde z)\right|^2 \le~& L_i^2 + M_i^2+N_i^2+L_iM_i+L_iN_i+M_iN_i  \\
        \le~& (\log np)^{-2\gamma}+(2\tilde C_1\tau^{-1/4})^2+(\frac{\log (np)\sqrt{np}}{np}m(\tilde z))^2+\cdots \\
        +~& (\log np)^{-\gamma}\cdot2\tilde C_1\tau^{-1/4}+(\log np)^{-\gamma}\cdot\frac{\log (np)\sqrt{np}}{np}m(\tilde z) +\cdots \\+~&2\tilde C_1\tau^{-1/4}\cdot\frac{\log (np)\sqrt{np}}{np}m(\tilde z)~.
    \end{align*}
    The dominating term depends on the value of $\gamma>0$, the sparsity of $p$, and the sparsity of $\tau$. Thus, for values of $n$ large enough, we see there is a constant $C_5>0$ such that
    \begin{align}
        \|\mathcal L \|_F^2 \le \| \Theta \|_F^2 \left(C_5r^2\max\{(\log(np))^{-2\gamma}, ~ \tau^{-1/2}\} \right) := \Lambda^2 ~.\label{Capital_Lambda}
    \end{align}
\end{proof}

\subsection{Non-distinct Signals}\label{equal_signals}
In this section, we consider the case where the matrix $\Theta$ in (\ref{model}) has signal $\theta_i$ repeated $r_0$ times for different signal vectors.
Theorem \ref{eigenvalue_theorem} would hold as there is no reliance on the signals being distinct. However, Theorem \ref{eigenvector_theorem} needs to be updated to a slightly altered result:
\begin{theorem}
Considering the model (\ref{model}) but with $\Theta$ having signal $\theta_i$ repeated $r_0$ times, we have, under Assumption~\ref{p_assumption}, the following behavior for any unit-norm eigenvector $u_i(X)$ corresponding to the eigenvalue generated by $\theta_i$ in Theorem \ref{eigenvalue_theorem}:  
\begin{align*}
    |\langle u_i(X),v_j \rangle| \xrightarrow{\P} 
    \begin{cases}
        (w_i)_j & \theta _i > 1 \\
        0 & \theta _i \leq 1 \text{ or } j \not\in \{i, i+1, \dots i+r_0 \}
    \end{cases}
\end{align*}
where $v_j$'s are defined by \eqref{eq:def_spikes}, for all $i,j\in [r]$, and $w_i\in \text{Span}_{l=i}^{i+r_0}\{e_{l} \}$ with $||w_i||^2=1-\frac{1}{\theta_i^2}$. More specifically, for $\theta_i>1$ and for any $\gamma>0$, there exist some positive constants $c_1, C_2, C_3, C_5, C_8,$ and $D$ such that
\begin{align*}
    \mathbb P \left(  \left| \sum_{l=i}^{i+r_0}\langle u_i(X), \frac{1}{\sqrt{np}}v_l \rangle^2-(1-\frac{1}{\theta_i^2})\right| \lesssim \max\{(\log(np))^{-\gamma}, ~ \tau^{-1/4}\} \right) \ge 1-\bm f_r(n)~,
\end{align*}
and, for $j\not \in \{ i, i+1, \dots i+r_0\}$, 
\begin{align*}
    \mathbb P \left( \left | \langle u_i(X), \frac{1}{\sqrt{np}}v_j \rangle^2\right| \lesssim \Lambda^2 \right) \ge 1-\bm f_r(n)~,
\end{align*}
 hold, where $\bm f_r(n)$ is defined by \eqref{eq:rate_func} and
$$\Lambda = \|\Theta\|_F\sqrt{C_5r^2\max\{(\log(np))^{-2\gamma}, ~ \tau^{-1/2}\}}~.$$
\end{theorem}
\begin{proof}
Without loss of generality, say $v_1, v_2, \dots, v_{r_0}$ all have signal $\theta_1$. Then, \eqref{matrix_convergence} at the end of the proof of Theorem \ref{eigenvalue_theorem} would be replaced with
\begin{align}\label{eigenvalue_matrix_conv}
    I+\frac{1}{np}V^TRV\Theta \to\text{diag}\left(0,\dots, 0, 1-\frac{\theta _2}{\theta _1}, ...1-\frac{\theta _r}{\theta _1}\right)~,
\end{align}
where the matrix leads with $r_0$ zeros on the diagonal. We again choose $(\tilde z_1,u_1)$ to be a normalized eigenpair for $X$ corresponding to the recovered eigenvalue $z_1=\theta_1+\frac{1}{\theta_1}$. The proof now follows in a similar manner to the distinct spike case (Theorem \ref{eigenvector_theorem}). 

Following the same argument, $V^Tu_1$ is in the kernel of the former term of \eqref{eigenvalue_matrix_conv}, so the limit point of $V^Tu_1$ is in the kernel of the latter matrix $\text{diag}\left(0,\dots, 0, 1-\frac{\theta _2}{\theta _1}, ...1-\frac{\theta _r}{\theta _1}\right)$, i.e. $$\frac{1}{\sqrt{np}}V^Tu_1\to w_1\in\text{span}\{e_1, \dots, e_{r_0}\}~.$$ Thus, $\langle u_1, \frac{1}{\sqrt{np}}v_j \rangle \to 0$ for $j>r_0$ and $\langle u_1, \frac{1}{\sqrt{np}}v_j \rangle \to (w_1)_j$ for $j\le r_0$. Next, one would follow the same argument made in the proof of Theorem \ref{eigenvector_theorem} to achieve slightly altered convergence bounds similar to \eqref{eigv_fluc_1} and \eqref{eigv_fluc_2}:
\begin{align*}
    |\langle u_1, \frac{1}{\sqrt{np}}v_j \rangle|= O(1) \text{ for } j\le r_0 \text{ and } |\langle u_1, \frac{1}{\sqrt{np}}v_j\rangle|= O(\Lambda) \text{ for } j>r_0~, \text{ and}
\end{align*}
\begin{align*}
    |f'_{n,i,i}(z)-m'(z)|= O(\max\{(\log(np))^{-\gamma}, ~ \tau^{-1/4}\}) \text{ and } |f'_{n,i,j}(z)|= O((\log np)^{-\gamma}) ~.
\end{align*}
Then, one would repeat the decomposition of $||u_1||^2$ and use the normalization of $u_1$ to find the same result: 
$$\| \frac{1}{\sqrt{np}}\tilde \Theta V^Tu_1 \|_2^2=1+O(\Lambda(\log np)^{-\gamma })~,$$ 
where $\tilde \Theta$ is suitably updated to reflect the repeat signals. Then, the final fluctuation would conclude in a very similar manner with minor changes: 
{\small
\begin{align*}
    \left|\theta_1^2\sum_{j=1}^{r_0}\langle u_1, \frac{1}{\sqrt{np}}v_j \rangle ^2-\frac{1}{m'(z_1)} \right|=~&\frac{1}{m'(z_1)}\left|\theta_1^2\sum_{j=1}^{r_0}(m'(z_1)+\mathcal M_j)\langle u_1, \frac{1}{\sqrt{np}}v_j \rangle ^2-\theta_1^2 \sum_{j=1}^{r_0}\mathcal M_j\langle u_1, \frac{1}{\sqrt{np}}v_j\rangle ^2-1 \right| \\
    \le~&O(1)\left|\| \frac{1}{\sqrt{np}}\tilde\Theta V^T u_1\|^2-1-\sum_{j=r_0+1}^r\theta_j^2(m'(z_1)+\mathcal M_j)\langle u_1,\frac{1}{\sqrt{np}}v_j\rangle^2 \right| + O(1)|\mathcal M_{\max} |\\
    \le~&O(1)\left|\| \frac{1}{\sqrt{np}}\tilde\Theta V^T u_1\|^2-1 \right| +O(1)\cdot\left(m'(z_1)+\mathcal M_{\max}\right)\cdot O(\Lambda^2) + O(1)|\mathcal M_{\max} |\\
    \le~& O(\Lambda (\log np)^{-\gamma})+O(\Lambda^2)+O(\max\{(\log(np))^{-\gamma}, ~ \tau^{-1/4}\}\Lambda^2)+O(\max\{(\log(np))^{-\gamma}, ~ \tau^{-1/4}\}) \\
    =~&O(\max\{(\log(np))^{-\gamma}, ~ \tau^{-1/4}\})~.
\end{align*}
} where we recall $\Lambda = \|\Theta\|_F\sqrt{C_5r^2\max\{(\log(np))^{-2\gamma}, ~ \tau^{-1/2}\}}~.$

Therefore, we recover eigenvector $u_1$, which has a nontrivial correlation with the Span$\{v_1, v_2, \dots v_{r_0} \}$, i.e. 
\begin{align*}
    \begin{bmatrix}
        \langle u_1, \frac{1}{\sqrt{np}}v_1 \rangle \\ \langle u_1, \frac{1}{\sqrt{np} }v_2 \rangle \\ \vdots \\ \langle u_1, \frac{1}{\sqrt{np}}v_r \rangle
    \end{bmatrix}
    = \frac{1}{\sqrt{np}}V^Tu_1 \to w_1 \text{ for } \theta _1>1~,
\end{align*}

where $||w_1||=\sqrt{1-\frac{1}{\theta_1^2}}>0$.

\end{proof}

\subsection{Varying Spike Sparsities}

In this section, we consider the case in which each spike vector $v_i$ has a different sparsity $p_i$. We claim Theorems \ref{eigenvalue_theorem} and \ref{eigenvector_theorem} hold when Assumption \ref{p_assumption} is updated:
\begin{assumption}\label{different_p_assumption}
    For the noise sparsity, assume $q\gtrsim \frac{(\log(n))^{1+\tau}}{n}$ for some $\tau>0$. For the spike sparsity, assume $p_i\gg \frac{1}{n}$ for all $i\in [r]$ and $$\frac{np_{\min}^2}{p_{\max}\cdot(\log(np_{\min}))^{\Theta(1)}} \to\infty~.$$ Thus, we have $np_i\to\infty$ for all $i\in [r]$ and $nq\to\infty$.
\end{assumption}
The model now needs updating to remove scaling on the outside term: 

$$\Theta = \frac{1}{n}\begin{bmatrix}\frac{1}{p_1}\theta_1 & & &  \\
& \frac{1}{p_2}\theta_2 & &  \\
& & \ddots & \\
& & & \frac{1}{p_r}\theta_r
\end{bmatrix}~.$$
The preliminary lemmas \ref{semicirc_law_sparse_lemma}, \ref{lemma_2}, and \ref{fixed_eig_noise_lemma} covered in Section \ref{prelim_lemmas} would remain unchanged as they do not depend on $p$. Lemma \ref{support_size_independent} would remain unchanged. 




Within the main proof of Theorem \ref{eigenvalue_theorem}, there are two key results we must adjust for the diagonal terms. First,  (\ref{resolv_exp_control}) would adjust to: 
\begin{align*}
    \mathbb{P}\left(\exists i \in [r] : |\frac{1}{n p_i+k_i}\text{tr}(R_n(z)_{\text{supp}(v_i)})-m(z)|> \tilde C_1\tau^{-1/4} \right)\le rn^{-D}~,
\end{align*}
where the bound did not change as it does not depend on $p$. However, for the concentration term \eqref{allspikesdiagonal}, we need to adjust the Hanson-Wright Inequality to work for any scaling of $np_i$ as we use it in both the diagonal and off-diagonal concentration. As we do in the main proof, we see $\| R_{\text{supp}(v_i)}\|_F^2\le C_3C_2\cdot np_{\max}$ and thus
\begin{align*}
    \mathbb P \left(\frac{1}{np_j} \left| v_i^T R v_i - \mathbb{E}_{v_i} [ v_i^T R v_i]  \right| > \frac{t}{np_{\min}}\right) 
    \le~& \mathbb P \left( \frac{1}{np_{\min}}\left|v_i^T R v_i - \mathbb{E}_{v_i} [ v_i^T R v_i] \right| > \frac{t}{np_{\min}}\right) \\ \le~& 2 \exp\left(-c_1\min \left\{ \frac{t^2}{K^4\cdot C_3C_2 np_{\max} }, \frac{t}{K^2\sqrt{C_2}} \right\} \right)~.
\end{align*}
Thus, we choose $t=\frac{np_{\min}}{(\log(np_{\min})^\gamma}$ and perform a union bound to see 
\begin{align*}
    \mathbb P \left(\exists \{i,j\}\in[r]^2:\frac{1}{np_j} \left| v_i^T R v_i - \mathbb{E}_{v_i} [ v_i^T R v_i]  \right| > \frac{1}{(\log(np_{\min}))^\gamma} \right) \le r(r-1)\exp\left(-\frac{c_1}{C_3C_2K^4}\frac{p_{\min}}{p_{\max}}\frac{np_{\min}}{(\log(np_{\min}))^{2\gamma}} \right)~,
\end{align*}
which thus introduces our condition of $$\frac{p_{\min}}{p_{\max}}\frac{np_{\min}}{(\log(np_{\min}))^{2\gamma}} \to\infty~.$$

As for the off-diagonal concentration, the Hanson-Wright Inequality would need to be slightly updated in a similar fashion, taking the new scaling into consideration. Based on the updated model, \eqref{offdiagtilderewrite} would update to a similar decomposition of the term $\frac{2}{np_i}v_i^TRv_j$: 
\begin{align}
    \frac{1}{np_i}2v_i^TRv_j
    :=~& \frac{1}{np_i}\left[\tilde v^T \tilde R \tilde v -v_i^TRv_i-v_j^TRv_j \right] \notag \\
    =~& \frac{1}{np_i}\left[\tilde v^T \tilde R \tilde v -\mathbb E [\tilde v^T \tilde R \tilde v]+\mathbb E [\tilde v^T \tilde R \tilde v] -v_i^TRv_i-v_j^TRv_j \right] \notag \\
    =~& \frac{1}{np_i} \left[ \tilde v^T \tilde R \tilde v -\mathbb E [\tilde v^T \tilde R \tilde v] + \mathbb E[ v_i^TRv_i]-v_i^TRv_i + \mathbb E[ v_j^TRv_j]-v_j^TRv_j \right] ~.\label{offdiagtilderewrite2}
\end{align}
For the last two terms, we would use the newly computed Hanson-Wright inequality above. 

Therefore, all that is left is to show the concentration of the first two terms of \eqref{offdiagtilderewrite2} with a Hanson-Wright Inequality once again. We would proceed as we do in the original general-$p$ proof and see $||\tilde R||_F^2\le 4C_3C_2np_{\max},$ which would bring us to the Hanson-Wright Inequality: 
\begin{align*}
    \mathbb P \left( \exists \{i,j\} \in [r]^2 : \frac{1}{np_i}\left| \tilde v^T R \tilde v - \mathbb E [\tilde v^T R \tilde v]\right| > \frac{1}{(\log(np_{\min}))^\gamma}\right) \le r(r-1) \exp\left(-\frac{c_1}{4C_3C_2K^4}\frac{p_{\min}}{p_{\max}}\frac{np_{\min}}{(\log(np_{\min}))^{2\gamma}} \right)~.
\end{align*}
Based on the updated bounds above, the high probability events $\mathcal A_3, \mathcal A_4$, and $\mathcal A_5$ would be updated accordingly for their use in the fluctuation results and proof of Theorem \ref{eigenvector_theorem}. For the latter, the proof of this theorem follows in the exact same manner as written in the case of varying $p_i$ values. Theorem \ref{spike_two_norm_control} would also be adapted in a similar way, with bounds reflecting the varying $p_i$ values. Therefore, under the new Assumption \ref{different_p_assumption}, Theorem \ref{eigenvalue_theorem} and Theorem \ref{eigenvector_theorem} would remain the exact same with only updates to the fluctuation error bounds: 
\begin{align*}
    \bm f_r(n):=\max\left\{r(r-1)\exp\left(-\frac{c_1}{4C_3C_2K^4}\frac{p_{\min}}{p_{\max}}\frac{np_{\min}}{(\log(np_{\min}))^{2\gamma}} \right), ~ rn^{\max\{-2,-D\}}, ~\exp(-C_8nq)\right\}
\end{align*}
and
\begin{align*}
    \Lambda = \|\Theta\|_F\sqrt{C_5r^2\max\{(\log(np_{\min}))^{-2\gamma}, ~ \tau^{-1/2}\}}~,
\end{align*}
where the latter was updated by trivially upper-bounding by the minimum $p_{\min}$.